\documentclass[12pt, reqno]{amsart}
\usepackage{amsmath, amsthm, amscd, amsfonts, amssymb, graphicx, color}
\usepackage[bookmarksnumbered, colorlinks, plainpages]{hyperref}
\usepackage{cite}
\newtheorem*{theoA}{Theorem A}
\newtheorem*{theoB}{Theorem B}
\newtheorem*{theoC}{Theorem C}
\newtheorem*{theoD}{Theorem D}
\newtheorem*{theoE}{Theorem E}

\newtheorem{theo}{Theorem}[section]
\newtheorem{lem}{Lemma}[section]

\newtheorem{exm}{Example}[section]

\newtheorem{ques}{Question}[section]
\newcommand{\ol}{\overline}
\newcommand{\be}{\begin{equation}}
	\newcommand{\ee}{\end{equation}}
\newcommand{\beas}{\begin{eqnarray*}}
	\newcommand{\eeas}{\end{eqnarray*}}
\newcommand{\bea}{\begin{eqnarray}}
	\newcommand{\eea}{\end{eqnarray}}
\numberwithin{equation}{section}

\begin{document}
	\setcounter{page}{1}
	
\title[existence of entire solutions.....]{Characterization of entire solutions of systems of quadratic trinomial difference and partial differential difference equations in $\mathbb{C}^n$}
	
	\author[G. Haldar]{Goutam Haldar}
	
	\address{Department of Mathematics, Malda College, Malda, West Bengal 732101, India.}
	\email{goutamiit1986@gmail.com}
	
	
	
	\date{Received: xxxxxx; Revised: yyyyyy; Accepted: zzzzzz.
		\newline\indent $^{*}$ Corresponding author}
	
\begin{abstract}
In this paper we establish some results about the existence and precise forms of finite order entire solutions of some systems of quadratic trinomial functional equations one of which in $\mathbb{C}^n$, $n\in\mathbb{N}$ and other two in $\mathbb{C}^2$. Our results are the generalizations and improvements of the previous theorems given by Xu-Cao \cite{Xu & Cao & 2018,Xu & Cao & 2020} and Xu-Jiang \cite{Xu Jiang RCSM 2022}. Moreover, we exhibit some examples in support of our claims.
\end{abstract} \maketitle
	
\section{Introduction}
It has a long history on the study of entire and meromorphic solutions of Fermat type equations by employing Nevanlinna theory as a tool. For detail study, we insist the readers to go through \cite{Montel & Paris & 1927,Gross & Bull. Amer. & 1966, Gross & Amer. Math. & 1966,Iyer & J. Indian. Math. Soc. & 1939,Chiang & Feng & Trans. Am. & 2009,Chen & JMMA & 2011,Chen & Acta Mathematica Scientia & 2012,Zheng Tu & JMMA & 2011,Yang & Li & 2004, Tailor & Wiles & 1995,Wiles & Ann. Math. 1995,Montel & Paris & 1927} and the references therein. We assume the readers are familiar with the notations and basic definitions of Nevanlinna theory (see \cite{Laine & 1993,Hayman & 1964}). In 1939, Iyer \cite{Iyer & J. Indian. Math. Soc. & 1939} proved that the entire solutions of Fermat type equation \bea\label{e1.1} f(z)^{2}+g(z)^2=1\eea are $f=\cos a(z)$,  $g=\sin a(z)$, where $a(z)$ is an entire function.\vspace{1.2mm}
\par Utilizing difference Nevanlinna theory of meromorphic functions developed by Halburd-Korhonen \cite{Halburd & Korhonen & 2006, Halburd & Korhonen & Ann. Acad. & 2006}, and Chiang-Feng \cite{Chiang & Feng & 2008}, independently, Liu \textit{et al.} \cite{Liu & Cao & Arch. Math. & 2012}, in $2012$, proved that any finite order transcendental entire solutions of \bea\label{e1.2} f(z)^2+f(z+c)^2=1,\eea has the form $f(z)=\sin(az+b)$, where $b$ is a constant and $a=(4k+1)\pi/2c$, with $k$ is an integer. In the same paper, they also proved that any finite order transcendental entire solution of \bea\label{e1.3}f^{\prime}(z)^2+f(z+c)^2=1\eea has the form $f(z)=\sin(z\pm Bi)$, where $B$ is a constant and $c=2p\pi$ or $c=(2p+1)\pi$, $p\in\mathbb{Z}$, whereas any transcendental entire solutions with finite order of \bea\label{e1.4}f^{\prime}(z)^2+[f(z+c)-f(z)]^2=1,\eea must take the form $f(z)=12\sin(2z+bi)$, where $c=(2p+1)\pi$, $p$ being an integer, and $b$ is a constant.\vspace{1.2mm}
\par Although there are many important and interesting results on the existence of entire and meromorphic solutions of complex difference as well as differential-difference equations (see \cite{Liu & Cao & EJDE & 2013,Liu & Cao & Arch. Math. & 2012,Liu & Yang & 2016,Liu & JMAA & 2009,Liu & Yang & CMFT & 2013,Liu-Dong & EDJE & 2015,Liu-Gao & Sci. Math. & 2019,Tang & Liao & 2007}), only a few number of results are there on the existence and forms of entire and meromorphic solutions of quadratic trinomial functional equation. To the best of our knowledge, Saleeby \cite{Saleeby & Aeq. Math. & 2013} first investigated the existence and forms of entire and meromorphic solutions of quadratic trinomial functional equation \bea\label{e1.3a} f^2+2wfg+g^2=1,\eea where $w$ is a constant.  Observe that \eqref{e1.1} can be seen as the case of $w=0$. For $w\neq0$, Saleeby \cite{Saleeby & Aeq. Math. & 2013} proved the following classical results about the Fermat type functional equations \eqref{e1.3a}.
\begin{theoA}{\em\cite{Saleeby & Aeq. Math. & 2013}}
Let $\alpha\in\mathbb{C}$ with $\alpha^2\neq0,1$. Then any transcendental entire solution of \eqref{e1.3a} has the form $f(z)=\dfrac{1}{\sqrt{2}}\left(\dfrac{\cos(h(z))}{\sqrt{1+w}}+\dfrac{\sin(h(z))}{\sqrt{1-w}}\right)$ and $\dfrac{1}{\sqrt{2}}\left(\dfrac{\cos(h(z))}{\sqrt{1+w}}-\dfrac{\sin(h(z))}{\sqrt{1-w}}\right)$, where $h$ is entire $\mathbb{C}^n$. The meromorphic solutions of \eqref{e1.3a} is of the form $f(z)=\dfrac{\alpha_1-\alpha_2\beta(z)^2}{(\alpha_1-\alpha_2)\beta(z)}$ and $g(z)=\dfrac{1-\beta(z)^2}{(\alpha_1-\alpha_2)\beta(z)}$, where $\beta(z)$ is meromorphic in $\mathbb{C}^n$ and $\alpha_1=-w+\sqrt{w^2-1}$, $\alpha_2=-w-\sqrt{w^2-1}$.\end{theoA}
\par By utilizing Theorem A, in 2016, Liu-Yang \cite{Liu & Yang & 2016} obtained more precise properties of entire and meromorphic solutions of \bea\label{e1.7} f(z)^2+2wf(z)f^{\prime}(z)+f^{\prime}(z)^2=1\eea and \bea\label{e1.8} f(z)^2+2wf(z)f(z+c)+f(z+c)^2=1,\;\; w\neq0,\pm1.\eea 
\par To the best of our knowledge, there are only a few number of results on the system of complex difference and differential-difference equations in the literature (see \cite{Gao & Acta Math. Sinica & 2016}). It was Gao \cite{Gao & Acta Math. Sinica & 2016}, who, in $2016$, first discussed the existence and form of transcendental entire solutions of system of differential difference equations and obtained the result as follows: If $(f,g)$ be a pair of finite order transcendental entire solution of the system of differential-difference equations \bea\label{e1.5}\begin{cases}
f^{\prime}(z)^2+g(z+c)^2=1,\\g^{\prime}(z)^2+f(z+c)^2=1.\end{cases}\eea Then $(f,g)$ must be of the form \beas (f,g)=(\sin(z-ib),\sin(z-ib_1)),\;\;\text{or}\;\; (\sin(z+ib),\sin(z+ib_1)),\eeas where $b,b_1$ are constants and $c=k\pi$, $k$ is a integer.\vspace{1.2mm}
\section{Solutions of Fermat type equations in several complex variables}
\par Recently, the study of entire and meromorphic solutions of Fermat type partial differential equations in several complex variables has received considerable attention in the literature (see \cite{Li & Arch. Math. & 2008,Li & Int. J. & 2004,Saleeby & Analysis & 1999,Saleeby & Complex Var. & 2004,Li Complex Var & 1996}). The study partial differential equations which is a generalizations of the well-known eikonal equation in real variable case has a long history. We insist the readers to go through \cite{Courtant Hilbert & 1962,Garabedian & Wiley & 1964,Rauch & 1991} and the references therein. In 1995, Khavinson pointed out that in $\mathbb{C}^2$, the entire solution of the Fermat type partial differential equations $f_{z_1}^2+f_{z_2}^2=1$ must necessarily be linear. After the development of the difference Nevanlinna theory in several
complex variables, specially the difference version of logarithmic derivative lemma (see \cite{Cao & Korhonen & 2016,Biancofiore & Stoll & 1981,Cao & Xu & Ann. Math. Pure Appl. & 2020,Korhonen & CMFT & 2012}), many Researchers started studying the existence and the precise form of entire and meromorphic solutions of different variants of Fermat type difference and partial differential difference equations, and obtained very remarkable and interesting results (see \cite{Xu & Cao & 2018,Xu & Cao & 2020,Haldar & 2023 & Mediterr,Haldar-Ahamed & Anal. Math. Phys. 2022,Xu Haldar & EJDE & 2023,ronkin & AMS & 1971,Stoll & AMS & 1974,Xu-Liu-Li-JMAA-2020}).\vspace{1.2mm}
\par Hereinafter, we denote by $z+w=(z_1+w_1,z_2+w_2,\ldots,z_n+w_n)$, where $z=(z_1,z_2,\ldots,z_n),w=(w_1,w_2,\ldots,w_n)\in\mathbb{C}^n$, $n$ being a positive integer.\vspace{1.2mm}
\par In 2018, Xu-Cao \cite{Xu & Cao & 2018} established the existence of entire solutions for some Fermat type complex difference and partial differential difference
equations with several variables. We list some of them below.
\begin{theoB}{\em\cite{Xu & Cao & 2018}}
Let $c=(c_1,c_2,\ldots,c_n)\in \mathbb{C}^n\setminus\{(0,0,\ldots,0)\}$. If $f:\mathbb{C}^n\rightarrow\mathbb{P}^{1}(\mathbb{C})$ be an entire solution with finite order of the Fermat type difference equation \bea\label{e2.1a} f(z)^2+f(z+c)^2=1,\eea the $f(z)$must assume the form $f(z)=\cos(L(z)+B)$, where $L$ is a linear function of the form $L(z)=a_1z_1+\cdots+ a_nz_n$ on $\mathbb{C}^n$ such that $L(c)=-\pi/2-2k\pi$ $(k\in\mathbb{Z})$, and $B$ is a constant on $\mathbb{C}$.\end{theoB}
\begin{theoC}{\em\cite{Xu & Cao & 2018}}
Let $c=(c_1,c_2)\in\mathbb{C}^2$. If $f(z)$ be transcendental entire solution with finite order of the Fermat type partial differential-difference equation \bea\label{e2.2a} \left(\frac{\partial f(z_1,z_2)}{\partial z_1}\right)^2+f(z_1+c_1,z_2+c_2)^2=1,\eea $f$ must be of the form  $f(z_1,z_2)=\sin(Az_1+Bz_2+\phi(z_2))$, Where $A, B$ are constant on $\mathbb{C}$ satisfying $A^2=1$ and $Ae^{i(Ac_1+Bc_2)}=1$, and $\phi(z_2)$ is a polynomial in one variable $z_2$ such that $\phi(z_2)\equiv \phi(z_2 + c_2)$. In the special case whenever $c_2\neq0$, we have $f(z_1,z_2)=\sin (Az_1+Bz_2+ Constant)$.\end{theoC}
\par Very recently, corresponding to Equations \eqref{e1.6} and \eqref{e1.7}, Xu-Jiang \cite{Xu Jiang RCSM 2022} considered the following system of trinomial difference and differential-difference equations in $\mathbb{C}^2$ \bea\label{2}\begin{cases}f(z)^2+2wf(z)g(z+c)+g(z+c)^2=1\\g(z)^2+2wg(z)f(z+c)+f(z+c)^2=1\end{cases}\eea and 
\bea\label{3} \begin{cases} \left(\dfrac{\partial f}{\partial z_1}\right)^2+2w\dfrac{\partial f}{\partial z_1}g(z+c)+g(z+c)^2=1\vspace{1mm}\\\left(\dfrac{\partial g}{\partial z_1}\right)^2+2w\dfrac{\partial g}{\partial z_1}f(z+c)+f(z+c)^2=1,\end{cases}\eea and obtained the results as follows. 
\begin{theoD}\em{\cite{Xu Jiang RCSM 2022}}
Let $c=(c_1,c_2)\in\mathbb{C}^2\setminus\{(0,0)\}$. Then any pair of finite order transcendental entire solutions of \eqref{2} must assume one of the following forms:
\begin{enumerate}
\item [(i)] \beas f(z)=\frac{1}{\sqrt{2}}\left(\frac{\cos(\gamma(z)+b_1)}{\sqrt{1+w}}+\frac{\sin(\gamma(z)+b_1)}{\sqrt{1-w}}\right),\eeas \beas g(z)=\frac{1}{\sqrt{2}}\left(\frac{\cos(\gamma(z)+b_2)}{\sqrt{1+w}}+\frac{\sin(\gamma(z)+b_2)}{\sqrt{1-w}}\right),\eeas where $\gamma(z)=L(z)+H(s)$, $L(z)=a_1z_1+a_2z_2$, $H(s)$ is a polynomial in $s:=c_2z_1-c_1z_2$, $a_i,b_i,c_i$, $i=1,2$ satisfy the relations \beas e^{2iL(c)}=\frac{-w+\sqrt{w^2-1}}{-w-\sqrt{w^2-1}},\;\; e^{2i(b_1-b_2)}=1.\eeas
\item [(ii)] \beas f(z)=\frac{1}{\sqrt{2}}\left(\frac{\cos(\gamma(z)+b_1)}{\sqrt{1+w}}+\frac{\sin(\gamma(z)+b_1)}{\sqrt{1-w}}\right),\eeas \beas g(z)=\frac{1}{\sqrt{2}}\left(\frac{\cos(\gamma(z)+b_2)}{\sqrt{1+w}}-\frac{\sin(\gamma(z)+b_2)}{\sqrt{1-w}}\right),\eeas where $\gamma(z)$ is the same as in (i), and $a_i,b_i,c_i$, $i=1,2$ satisfy the relations \beas e^{2iL(c)}=1,\;\; e^{2i(b_1-b_2)}=1.\eeas\end{enumerate}\end{theoD}
\begin{theoE}\em{\cite{Xu Jiang RCSM 2022}}
Let $c=(c_1,c_2)\in\mathbb{C}^2\setminus\{(0,0)\}$ with $c_2\neq0$. Then any pair of finite order transcendental entire solutions of \eqref{3} must assume one of the following forms:	\begin{enumerate}\item [(i)] \beas f(z)=\frac{1}{\sqrt{2}}\left(\frac{\cos(L(z)-L(c)+b_1)}{\sqrt{1+w}}+\frac{\sin(L(z)-L(c)+b_1)}{\sqrt{1-w}}\right),\eeas \beas g(z)=\frac{1}{\sqrt{2}}\left(\frac{\cos(L(z)-L(c)+b_2)}{\sqrt{1+w}}+\frac{\sin(L(z)-L(c)+b_2)}{\sqrt{1-w}}\right),\eeas where $L(z)=a_1z_1+a_2z_2$, $a_i,b_i,c_i$, $i=1,2$ satisfy the relations \beas a_1^2=1,\;\;e^{2iL(c)}=\frac{-w+\sqrt{w^2-1}}{-w-\sqrt{w^2-1}},\;\; e^{2i(b_1-b_2)}=1.\eeas	\item [(ii)] \beas f(z)=\frac{1}{\sqrt{2}}\left(\frac{\cos(L(z)-L(c)+b_1)}{\sqrt{1+w}}+\frac{\sin(L(z)-L(c)+b_1)}{\sqrt{1-w}}\right),\eeas \beas g(z)=\frac{1}{\sqrt{2}}\left(\frac{\cos(L(z)-L(c)+b_2)}{\sqrt{1+w}}-\frac{\sin(L(z)-L(c)+b_2)}{\sqrt{1-w}}\right),\eeas where $L(z)$ is the same as in (i), and $a_i,b_i,c_i$, $i=1,2$ satisfy the relations \beas a_1^2=1,\;\; e^{2iL(c)}=1,\;\; e^{2i(b_1-b_2)}=1.\eeas\end{enumerate}\end{theoE}
\par Observing Theorems B and E, it is inevitable to ask the following open questions. 
\begin{ques}
What can be said about the entire solutions of \eqref{2} when one replace the right hand sides of each equation by $e^{g_1(z)}$ and $e^{g_2(z)}$, respectively, where $g_1(z), g_2(z)$ are any two polynomials in $\mathbb{C}^n$, $n$ being an integer$?$ 
\end{ques}
\begin{ques} What can be said about the existence and precise form of transcendental entire solutions of \eqref{3} when one replace (a) first order partial derivative by $k$-th order partial derivative for any $k\in\mathbb{N}$ and (b) right hand sides of each of the equations by $e^{g_1(z_1,z_2)}$ and $e^{g_2(z_1,z_2)}$, respectively, where $g_1(z_1,z_2), g_2(z_1,z_2)$ are any two polynomials in $\mathbb{C}^2$$?$\end{ques}
\begin{ques}
What can be said about the existence and precise form of transcendental entire solutions of \bea\label{e1.6} \begin{cases}		\left[\dfrac{\partial^k f}{\partial z_1^k}\right]^2+2w\dfrac{\partial^k f}{\partial z_1^k}[g(z+c)-g(z)]+[g(z+c)-g(z)]^2=1,\vspace{1mm}\\\left[\dfrac{\partial^k g}{\partial z_1^k}\right]^2+2w\dfrac{\partial^k g}{\partial z_1^k}[f(z+c)-f(z)]+[f(z+c)-f(z)]^2=1?\end{cases}\eea 
\end{ques}
\section{Results and Examples}
\par Inspired by the above results and questions, we investigate the existence and form of transcendental entire solutions of the system of equations \bea\label{e2.1}\begin{cases}f(z)^2+2wf(z)g(z+c)+g(z+c)^2=e^{g_1(z)}\\g(z)^2+2wg(z)f(z+c)+f(z+c)^2=e^{g_2(z)},\end{cases}\eea
\bea\label{e2.3} \begin{cases} \left[\dfrac{\partial^k f}{\partial z_1^k}\right]^2+2w\dfrac{\partial^k f}{\partial z_1^k}g(z+c)+g(z+c)^2=e^{g_1(z_1,z_2)}\vspace{1mm}\\\left[\dfrac{\partial^k g}{\partial z_1^k}\right]^2+2w\dfrac{\partial^k g}{\partial z_1^k}f(z+c)+f(z+c)^2=e^{g_2(z_1,z_2)},\end{cases}\eea where $g_1(z_1,z_2)$, $g_2(z_1,z_2)$ are any two polynomials and $k\in\mathbb{N}$ and Eq. \eqref{e1.6}. Our results are the generalizations of Theorems B--C to a large extent.\vspace{1.2mm}
\par Before we state our main results, let us denote two non-zero complex numbers $A_1,A_2$ as follows. \beas A_1=\frac{1}{2\sqrt{1+w}}+\frac{1}{2i\sqrt{1-w}}\;\;\text{and}\;\;A_2=\frac{1}{2\sqrt{1+w}}-\frac{1}{2i\sqrt{1-w}}.\eeas

\par Before we state our main results, let us set the following. \bea\label{e2.5} \Phi(z)&&=\sum_{i_1=1}^{^nC_2}H_{i_1}^2(s_{i_1}^2)+\sum_{i_2=1}^{^nC_3}H_{i_2}^3(s_{i_2}^3)+\sum_{i_3=1}^{^nC_4}H_{i_3}^4(s_{i_3}^4)+\cdots\nonumber\\&&+\sum_{i_{n-2}=1}^{^nC_{n-1}}H_{i_{n-2}}^{n-1}(s_{i_{n-2}}^{n-1})+H_{n-1}^n(s_{n-1}^{n}),\eea
\bea\label{e2.6} \Psi(z)&&=\sum_{i_1=1}^{^nC_2}L_{i_1}^2(t_{i_1}^2)+\sum_{i_2=1}^{^nC_3}L_{i_2}^3(t_{i_2}^3)+\sum_{i_3=1}^{^nC_4}L_{i_3}^4(t_{i_3}^4)+\cdots\nonumber\\&&+\sum_{i_{n-2}=1}^{^nC_{n-1}}L_{i_{n-2}}^{n-1}(t_{i_{n-2}}^{n-1})+L_{n-1}^n(t_{n-1}^{n}),\eea where $H_{i_1}^2$ is a polynomial in $s_{i_1}^2:=d_{i_1j_1}^2z_{j_1}+d_{i_1j_2}^2z_{j_2}$ with $d_{i_1j_1}^2c_{j_1}+d_{i_1j_2}^2c_{j_2}=0$, $1\leq i_1\leq \;^nC_2$, $1\leq j_1<j_2\leq n$, $H_{i_2}^3$ is a polynomial in $s_{i_2}^3:=d_{i_2j_1}^3z_{j_1}+d_{i_2j_2}^3z_{j_2}+d_{i_2j_3}^3z_{j_3}$ with $d_{i_2j_1}^3c_{j_1}+d_{i_2j_2}^3c_{j_2}+d_{i_2j_3}^3z_3=0$, $1\leq i_2\leq \;^nC_3$, $1\leq j_1<j_2<j_3\leq n\ldots$, $H_{i_{n-2}}^{n-1}$ is a polynomial in $s_{i_{n-2}}^{n-1}:=d_{i_{n-2}j_1}z_{j_1}+d_{i_{n-2}j_2}z_{j_2}+\cdots+d_{i_{n-2}j_{n-1}}z_{j_{n-1}}$ with $d_{i_{n-2}j_1}c_{j_1}+d_{i_{n-2}j_2}c_{j_2}+\cdots+d_{i_{n-2}j_{n-1}}c_{j_{n-1}}=0$, $1\leq i_{n-2}\leq \;^nC_{n-1}$, $1\leq j_1<j_2<\cdots<j_{n-1}\leq n$ and $H_{n-1}^{n}$ is a polynomial in $s_{n-1}^n:=d_{i_{n-1}1}z_1+d_{i_{n-1}2}z_2+\cdots+d_{i_{n-1}n}z_n$ with $d_{i_{n-1}1}c_1+d_{i_{n-1}2}c_2+\cdots+d_{i_{n-1}n}c_n=0$, $L_{i_1}^2$ is a polynomial in $t_{i_1}^2:=e_{i_1j_1}^2z_{j_1}+e_{i_1j_2}^2z_{j_2}$ with $e_{i_1j_1}^2c_{j_1}+e_{i_1j_2}^2c_{j_2}=0$, $1\leq i_1\leq \;^nC_2$, $1\leq j_1<j_2\leq n$, $H_{i_2}^3$ is a polynomial in $t_{i_2}^3:=e_{i_2j_1}^3z_{j_1}+e_{i_2j_2}^3z_{j_2}+e_{i_2j_3}^3z_{j_3}$ with $e_{i_2j_1}^3c_{j_1}+e_{i_2j_2}^3c_{j_2}+e_{i_2j_3}^3z_3=0$, $1\leq i_2\leq \;^nC_3$, $1\leq j_1<j_2<j_3\leq n\ldots$, $L_{i_{n-2}}^{n-1}$ is a polynomial in $t_{i_{n-2}}^{n-1}:=e_{i_{n-2}j_1}z_{j_1}+e_{i_{n-2}j_2}z_{j_2}+\cdots+e_{i_{n-2}j_{n-1}}z_{j_{n-1}}$ with $e_{i_{n-2}j_1}c_{j_1}+e_{i_{n-2}j_2}c_{j_2}+\cdots+e_{i_{n-2}j_{n-1}}c_{j_{n-1}}=0$, $1\leq i_{n-2}\leq \;^nC_{n-1}$, $1\leq j_1<j_2<\cdots<j_{n-1}\leq n$ and $L_{n-1}^{n}$ is a polynomial in $t_{n-1}^n:=e_{i_{n-1}1}z_1+e_{i_{n-1}2}z_2+\cdots+e_{i_{n-1}n}z_n$ with $e_{i_{n-1}1}c_1+e_{i_{n-1}2}c_2+\cdots+e_{i_{n-1}n}c_n=0$, where for each $k$, the representation of $s_i^k$ and $t_i^k$ in terms of the conditions of $j_1,j_2,\ldots,j_k$ are unique.\vspace{1.2mm}
\par Now, we list our main results as follows.
\begin{theo}\label{t1}
Let $c=(c_1,c_2,\ldots,c_n)\in\mathbb{C}^n$, $w\in\mathbb{C}$ such that $w^2\neq0,1$ and $g_1(z)$, $g_2(z)$ be any two polynomials in $\mathbb{C}^n$. If $(f(z),g(z))$ be a pair of transcendental entire solution with finite order of the system Fermat type difference equation \eqref{e2.1}, then one of the following must occur.
\begin{enumerate}
\item[\textbf{(i)}] $g_1(z)=L(z)+\Phi(z)+d_1$, $g_2(z)=L(z)+\Phi(z)+d_2$,
\beas f(z)=\frac{A_1\xi_1+A_2\xi_1^{-1}}{\sqrt{2}}e^{\frac{1}{2}[L(z)+\Phi(z)+d_1]},\;g(z)=\frac{A_1\xi_2+A_2\xi_2^{-1}}{\sqrt{2}}e^{\frac{1}{2}[L(z)+\Phi(z)+d_2]},\eeas where $L(z)=\sum_{j=1}^{n}a_jz_j$, $\Phi(z)$ is defined as in \eqref{e2.5} and  satisfy the relations \beas e^{L(c)}=\frac{(A_2\xi_2+A_1\xi_2^{-1})(A_2\xi_1+A_1\xi_1^{-1})}{(A_1\xi_1+A_2\xi_1^{-1})(A_1\xi_2+A_2\xi_2^{-1})},\eeas \beas e^{d_1-d_2}=\frac{(A_2\xi_2+A_1\xi_2^{-1})(A_1\xi_2+A_2\xi_2^{-1})}{(A_1\xi_1+A_2\xi_1^{-1})(A_2\xi_1+A_1\xi_1^{-1})},\eeas $a_j,d_1,d_2,\xi_1(\neq0), \xi_2(\neq0)\in\mathbb{C}$ for $j=1,2,\ldots,n$.\vspace{1.2mm}
\item[\textbf{(ii)}] \beas g_1(z)=L_1(z)+L_2(z)+\Phi(z)+\Psi(z)+d_1+d_3,\eeas \beas g_2(z)=L_1(z)+L_2(z)+\Phi(z)+\Psi(z)+d_2+d_4,\eeas \beas f(z)=\frac{1}{\sqrt{2}}\left[A_1e^{L_1(z)+\Phi(z)+d_1}+A_2e^{L_2(z)+\Psi(z)+d_3}\right],\eeas \beas g(z)=\frac{1}{\sqrt{2}}\left[A_1e^{L_2(z)+\Psi(z)+d_4}+A_2e^{L_1(z)+\Phi(z)+d_2}\right],\eeas where $L_1(z)=\sum_{j=1}^{n}a_jz_j$, $L_2(z)=\sum_{j=1}^{n}b_jz_j$, $\Phi(z)$, $\Psi(z)$ are defined in \eqref{e2.5} and $\eqref{e2.6}$, respectively, $a_j,b_j,d_i\in\mathbb{C}$ for $j=1,2,\ldots,n$, $i=1,\ldots,4$ such that $L_1(z)+\Phi(z)\neq L_2(z)+\Psi(z)$ and satisfy one of the following relations:\vspace{1.2mm}
\begin{enumerate}
\item [(a)] $e^{L_1(c)}=1$, $e^{L_2(c)}=1$, $e^{d_1-d_2}=1$ and $e^{d_3-d_4}=1$,\vspace{1.2mm}
\item [(b)] $e^{L_1(c)}=1$, $e^{L_2(c)}=-1$, $e^{d_1-d_2}=1$ and $e^{d_3-d_4}=-1$,\vspace{1.2mm}
\item [(c)] $e^{L_1(c)}=-1$, $e^{L_2(c)}=1$, $e^{d_1-d_2}=-1$ and $e^{d_3-d_4}=1$,\vspace{1.2mm}
\item [(d)] $e^{L_1(c)}=-1$, $e^{L_2(c)}=-1$, $e^{d_1-d_2}=-1$ and $e^{d_3-d_4}=-1$.\end{enumerate}\vspace{1.2mm}
\item[\textbf{(iii)}] $g_1(z)$ and $g_2(z)$ have the same form as in \textbf{(ii)}, \beas f(z)=\frac{1}{\sqrt{2}}\left[A_1e^{L_1(z)+\Phi(z)+d_1}+A_2e^{L_2(z)+\Psi(z)+d_3}\right],\eeas \beas g(z)=\frac{1}{\sqrt{2}}\left[A_1e^{L_1(z)+\Phi(z)+d_2}+A_2e^{L_2(z)+\Psi(z)+d_4}\right],\eeas where $L_1,L_2,\Phi,\Psi$ are same as in \textbf{(ii)} with $L_1(z)+\Phi(z)\neq L_2(z)+\Psi(z)$, and satisfy one of the following relations:\vspace{1.2mm}
\begin{enumerate}
\item [(a)] $e^{L_1(c)}=A_2/A_1$, $e^{L_2(c)}=A_1/A_2$, $e^{d_1-d_2}=1$ and $e^{d_3-d_4}=1$,\vspace{1.2mm}
\item [(b)] $e^{L_1(c)}=A_2/A_1$, $e^{L_2(c)}=-A_1/A_2$, $e^{d_1-d_2}=1$ and $e^{d_3-d_4}=-1$,\vspace{1.2mm}
\item [(c)] $e^{L_1(c)}=-A_2/A_1$, $e^{L_2(c)}=A_1/A_2$, $e^{d_1-d_2}=-1$ and $e^{d_3-d_4}=1$,\vspace{1.2mm}
\item [(d)] $e^{L_1(c)}=-A_2/A_1$, $e^{L_2(c)}=-A_1/A_2$, $e^{d_1-d_2}=-1$ and $e^{d_3-d_4}=-1$.\end{enumerate}
\end{enumerate}\end{theo}
\par The following examples show the existence of transcendental entire solutions with finite order of the system \eqref{e2.1}.
\begin{exm}
Let $n=2$, $w=3$, $\xi_1=\xi_2=1$, $c_1=c_2=\pi i$, $d_1=\pi i, d_2=-\pi i$, $g_1(z)=z_1+z_2+(z_1-z_2)^3+d_1$, $g_2(z)=z_1+z_2+(z_1-z_2)^3+d_2$. Then in view of the conclusion \textbf{(i)} of Theorem {\em\ref{t1}}, we can see that \beas (f(z),g(z))=\left(\frac{1}{2\sqrt{2}}e^{\frac{1}{2}[z_1+z_2+(z_1-z_2)^3+d_1]},\frac{1}{2\sqrt{2}}e^{\frac{1}{2}[z_1+z_2+(z_1-z_2)^3+d_2]}\right)\eeas is a solution of \eqref{e2.1}.\end{exm}
\begin{exm}
Let $w=2$, $n=3$, $L(z)=z_1+z_2+z_3$, $L_2(z)=z_1+z_2-z_3$, $c_1=\pi i, c_2=2\pi i, c_3=-\pi i$, $\Phi(z)=-\pi^2(z_1-z_2)^2+i\pi^3(z_1+z_2)^3+\pi^4(z_2-z_3)^4$, $\Psi(z)=\pi^4(z_1+2z_2-z_3)^2$. Choose $d_j\in\mathbb{C}$ such that $d_1-d_2=2\pi i$ and $d_3-d_4=4\pi i$. Then, in view of the conclusion \textbf{(ii)} of Theorem {\em\ref{t1}}, $(f(z),g(z))$, where \beas f(z)=\frac{1}{\sqrt{2}}\left[-\frac{3-\sqrt{3}}{6}e^{L_1(z)+\Phi(z)+d_1}+\frac{3+\sqrt{3}}{6}e^{L_2(z)+\Psi(z)+d_3}\right]\;\;\text{and}\eeas \beas g(z)=\frac{1}{\sqrt{2}}\left[-\frac{3-\sqrt{3}}{6}e^{L_2(z)+\Psi(z)+d_3}+\frac{3+\sqrt{3}}{6}e^{L_1(z)+\Phi(z)+d_1}\right]\eeas is a solution of \eqref{e2.1}.
\end{exm}
\begin{theo}\label{t2}
Let $k\in\mathbb{N}$, $c=(c_1,c_2)\in\mathbb{C}^2$, $w\in \mathbb{C}$ such that $w^2\neq0,1$, $g_1(z_1,z_2),g_2(z_1,z_2)$ be any two polynomials in $\mathbb{C}^2$. If $(f,g)$ is a pair of finite order transcendental entire solution of \eqref{e2.3}, then one of the following cases must occur.\vspace{1.2mm}
\begin{enumerate}
\item[\textbf{(i)}] $g_1(z_1,z_2)=\mathcal{L}(z)+\chi(z_2)+d_1$, $g_2(z_1,z_2)=\mathcal{L}(z)+\chi(z_2)+d_2$, \beas f(z)=\frac{A_2\xi_2+A_1\xi_2^{-1}}{\sqrt{2}}e^{\frac{1}{2}[\mathcal{L}(z)+\chi(z_2)-\mathcal{L}(c)+d_2]},\;g(z)=\frac{A_2\xi_1+A_1\xi_1^{-1}}{\sqrt{2}}e^{\frac{1}{2}[\mathcal{L}(z)+\chi(z_2)-\mathcal{L}(c)+d_1]},\eeas where $\mathcal{L}(z)=a_1z_1+a_2z_2$, $\chi(z_2)$ is a polynomial in $z_2$, only such that $\chi(z_2+c_2)=\chi(z_2)$, $\xi_j\neq0, a_j,d_j\in\mathbb{C}$, $j=1,2$, $a_1\neq0$, and satisfy the following: \beas \begin{cases}		e^{\mathcal{L}(c)}=\left(\dfrac{a_1}{2}\right)^{2k}\dfrac{(A_2\xi_2+A_1\xi_2^{-1})(A_2\xi_1+A_1\xi_1^{-1})}{(A_1\xi_1+A_2\xi_1^{-1})(A_1\xi_2+A_2\xi_2^{-1})},\vspace{1mm}\\e^{\frac{1}{2}[\mathcal{L}(c)+d_1-d_2]}=\left(\dfrac{a_1}{2}\right)^k\dfrac{A_2\xi_2+A_1\xi_2^{-1}}{A_1\xi_1+A_2\xi_1^{-1}}.\end{cases}\eeas In particular, if $c_2\neq0$, then $\chi(z_2)$ becomes linear in $z_2$.
\item[\textbf{(ii)}] \beas g_1(z_1,z_2)=L_1(z)+L_2(z)+\Phi_1(z_2)+\Psi_1(z_2)+d_1+d_3,\eeas \beas g_2(z_1,z_2)=L_1(z)+L_2(z)+\Phi_1(z_2)+\Psi_1(z_2)+d_2+d_4,\eeas \beas f(z_1,z_2)=\frac{A_2e^{L_2(z)+\Psi_1(z_2)-L_2(c)+d_4}+A_1e^{L_1(z)+\Phi_1(z_2)-L_1(c)+d_2}}{\sqrt{2}},\eeas\beas g(z_1,z_2)=\frac{A_2e^{L_1(z)+\Phi_1(z_2)-L_1(c)+d_1}+A_1e^{L_2(z)+\Psi_1(z_2)-L_2(c)+d_3}}{\sqrt{2}},\eeas where $L_1(z)=a_1z_1+a_2z_2$, $L_2(z)=b_1z_1+b_2z_2$, $\Phi_1, \Psi_1$ are two polynomials in $z_2$, only with $\Phi_1(z_2+c_2)=\Phi_1(z_2)$, $\Psi_1(z_2+c_2)=\Psi_1(z_2)$ such that $L_1(z)+\Phi_1(z_2)\neq L_2(z)+\Psi_1(z_2)$, $a_i,b_1,d_j\in\mathbb{C}$ for $i=1,2$, $j=1,2.3.4$, and satisfy one of the following relations:\vspace{1.2mm}
\begin{enumerate}\item [(a)] $e^{L_1(c)}=a_1^k$, $e^{L_2(c)}=b_1^k$, $e^{d_1-d_2}=1$ and $e^{d_3-d_4}=1$,\vspace{1.2mm}
\item [(b)] $e^{L_1(c)}=a_1^k$, $e^{L_2(c)}=-b_1^k$, $e^{d_1-d_2}=1$ and $e^{d_3-d_4}=-1$,\vspace{1.2mm}
\item [(c)] $e^{L_1(c)}=-a_1^k$, $e^{L_2(c)}=b_1^k$, $e^{d_1-d_2}=-1$ and $e^{d_3-d_4}=1$,\vspace{1.2mm}
\item [(d)] $e^{L_1(c)}=-a_1^k$, $e^{L_2(c)}=-b_1^k$, $e^{d_1-d_2}=-1$ and $e^{d_3-d_4}=-1$.\end{enumerate}\vspace{1.2mm}
\item[\textbf{(iii)}] The form of $g_1(z_1,z_2)$ and $g_2(z_1,z_2)$ are the same as in \textbf{(ii)}, \beas f(z_1,z_2)=\frac{A_2e^{L_1(z)+\Phi_1(z_2)-L_1(c)+d_2}+A_1e^{L_2(z)+\Psi_1(z_2)-L_2(c)+d_4}}{\sqrt{2}},\eeas\beas g(z_1,z_2)=\frac{A_2e^{L_1(z)+\Phi_1(z_2)-L_1(c)+d_1}+A_1e^{L_2(z)+\Psi_1(z_2)-L_2(c)+d_3}}{\sqrt{2}},\eeas where $L_1,L_2,\Phi_1,\Psi_1$ are the same as in \textbf{(ii)}, and satisfy one of the following conditions:\vspace{1.2mm}
\begin{enumerate}
\item [(a)] $e^{L_1(c)}=A_2a_1^k/A_1$, $e^{L_2(c)}=A_1b_1^k/A_2$, $e^{d_1-d_2}=1$ and $e^{d_3-d_4}=1$,\vspace{1.2mm}
\item [(b)] $e^{L_1(c)}=A_2a_1^k/A_1$, $e^{L_2(c)}=-A_1b_1^k/A_2$, $e^{d_1-d_2}=1$ and $e^{d_3-d_4}=-1$,\vspace{1.2mm}
\item [(c)] $e^{L_1(c)}=-A_2a_1^k/A_1$, $e^{L_2(c)}=A_1b_1^k/A_2$, $e^{d_1-d_2}=-1$ and $e^{d_3-d_4}=1$,\vspace{1.2mm}
\item [(d)] $e^{L_1(c)}=-A_2a_1^k/A_1$, $e^{L_2(c)}=-A_1b_1^k/A_2$, $e^{d_1-d_2}=-1$ and $e^{d_3-d_4}=-1$.\end{enumerate}	
\end{enumerate}\end{theo}
\par The following examples show the existence of transcendental entire solutions with finite order of the system \eqref{e2.3}.\vspace{1.2mm}
\begin{exm}Let $k=1$, $L(z)=z_1+z_2$, $\xi_1=1$, $\xi_2=-1$ and $w$ be a complex number such that $w\neq0,\pm 1$. Choose $c=(c_1,c_2)\in\mathbb{C}$ such that $e^{L(c)/2}=1/2$ and $d_1,d_2\in\mathbb{C}$ such that $e^{(d_1-d_2)/2}=-1$. Then, in view of the conclusion (i) of Theorem {\em\ref{t2}}, we easily see that $(f,g)$, where \beas f(z_1,z_2)=-\frac{2}{\sqrt{2(1+w)}}e^{\frac{1}{2}[z_1+z_2+d_2]}\;\;\text{and}\;\;g(z_1,z_2)=\frac{2}{\sqrt{2(1+w)}}e^{\frac{1}{2}[z_1+z_2+d_1]}\eeas is a solution of \eqref{e2.3} with $g_1(z_1,z_2)=z_1+z_2+d_1$ and $g_2(z_1,z_2)=z_1+z_2+d_2$.\end{exm}
\begin{exm}
Let $k=1$, $w=2$, $L_1(z)=z_1+z_2$, $L_2(z)=z_1-2z_2$, $c_1=8\pi i/3, c_2=-2\pi i/3$. Choose $d_j\in\mathbb{C}$ for $j=1,\ldots,4$ such that $d_1-d_2=2m\pi i$ and $d_3-d_4=2p\pi i$, $m,p\in\mathbb{Z}$. Then, in view of the conclusion (ii) of Theorem {\em\ref{t2}}, we easily see that $(f(z),g(z))$, where \beas f(z_1,z_2)=\frac{1}{6\sqrt{2}}\left[-(3-\sqrt{3})e^{z_1+z_2+d_2}+(3+\sqrt{3})e^{z_1-2z_2+d_4}\right]\;\;\text{and}\eeas \beas g(z_1,z_2)=\frac{1}{6\sqrt{2}}\left[(3+\sqrt{3})e^{z_1+z_2+d_1}-(3-\sqrt{3})e^{z_1-2z_2+d_3}\right]\eeas
\end{exm} is a solution of \eqref{e2.3} with $g_1(z_1,z_2)=2z_1-z_2+d_1+d_3$ and $g_2(z_1,z_2)=2z_1-z_2+d_2+d_4$.
\begin{theo}\label{t3}
Let $k$ be a positive integer, $c=(c_1,c_2)\in\mathbb{C}^2$, $w\in\mathbb{C}$ such that $w^2\neq0,1$ and $f(z), g(z)$ are transcendental entire functions of finite order not of $c$-periodic. If $(f(z),g(z))$ be a solution of \eqref{e1.6}, then one of the following conclusions must hold.\vspace{1.2mm}
\begin{enumerate}
\item [\textbf{(i)}]  \beas f(z)=\frac{A_1e^{L(z)+\Phi_1(z_2)+\gamma}+A_2e^{-(L(z)+\Phi_1(z_2)+\gamma)}}{\sqrt{2}\alpha^k}\;\; \text{and}\eeas \beas g(z)=\frac{A_1e^{L(z)+\Phi_1(z_2)+\gamma+\eta}+A_2e^{-(L(z)+\Phi_1(z_2)+\gamma+\eta)}}{\sqrt{2}\alpha^k}.\eeas where $L(z)=\alpha z_1+\beta z_2+\gamma$, $\Phi_1(z_2)$ is a polynomial in $z_2$ with $\Phi_1(z_2+c_2)=\Phi_1(z_2)$, $\alpha,\beta,\gamma,\eta\in\mathbb{C}$, and satisfy the following relations:\vspace{1.2mm}
\beas e^{L(c)}=\frac{A_2e^{\eta}}{(-1)^kA_1\alpha^k+A_2e^{\eta}},\;\; e^{2\eta}=1\;\;\text{and}\;\; \alpha^{k}=-\frac{(-1)^kA_1^2+A_2^2}{(-1)^kA_1A_2e^{\eta}}.\eeas In particular, if $c_2\neq0$, then the polynomial $\Phi_1(z_2)$ becomes linear in $z_2$.\vspace{1.2mm}
\item [\textbf{(ii)}] \beas f(z)=\frac{A_1e^{L(z)+\Phi_1(z_2)+\gamma}+A_2e^{-(L(z)+\Phi_1(z_2)+\gamma)}}{\sqrt{2}\alpha^k},\eeas \beas g(z)=\frac{A_1e^{-(L(z)+\Phi_1(z_2)+\gamma)+\eta}+A_2e^{L(z)+\Phi_1(z_2)+\gamma-\eta}}{\sqrt{2}\alpha^k},\eeas where $k$ is an even positive integer, $L(z),\Phi_1(z_2)$ are the same as in \textbf{(i)}, and satisfy the following relations:\vspace{1.2mm}
\beas e^{L(c)}=1+e^{-\eta}\alpha^k,\;\; e^{2\eta}=1\;\;\text{and}\;\; \alpha^{k}=-2e^{\eta}.\eeas
\end{enumerate}
\end{theo}
\par The following example shows that the existence of transcendental entire solutions of \eqref{e1.6} is precise.\vspace{1.2mm}
\begin{exm}
Let $k=2, w=2$ and choose $\alpha, \eta, c_1,c_2\in\mathbb{C}$ such that $\alpha^2=-2, \eta=2p\pi i$ and $c_1+c_2=(2m+1)\pi i$, $p,m\in\mathbb{Z}$. Then in view of the conclusion (ii) of Theorem \em{\ref{t3}}, we see that $(f,g)$, where \beas f(z_1,z_2)=\frac{1}{12\sqrt{2}}\left[(3-\sqrt{3})e^{\alpha z_1+z_2+1}-(3+\sqrt{3})e^{-(\alpha z_1+z_2+1)}\right]\;\;\text{and}\eeas  \beas g(z_1,z_2)=\frac{1}{12\sqrt{2}}\left[(3-\sqrt{3})e^{-(\alpha z_1+z_2+1)}-(3+\sqrt{3})e^{\alpha z_1+z_2+1}\right]\;\;\text{and}\eeas is a solution of \eqref{e1.6}.\end{exm}
\section{Proof of the main results}
\par Before we starting the proof of the main results, we present here some necessary lemmas which will play key role to prove the main results of this paper.
\begin{lem}\label{lem3.1}{\em\cite{Yan & Yi & 2003}}
Let $f_j\not\equiv0$ $(j=1,2\ldots,m;\; m\geq 3)$ be meromorphic functions such that $f_1,\ldots, f_{m-1}$ are not constants, $f_1+f_2+\cdots+ f_m=1$ and such that\beas \sum_{j=1}^{m}\left\{N_{n-1}\left(r,\frac{1}{f_j}\right)+(m-1)\ol N(r,f_j)\right\}< \lambda T(r,f_j)+O(\log^{+}T(r,f_j))\eeas holds for $j=1,\ldots, m-1$ and all $r$ outside possibly a set with finite logarithmic measure, where $\lambda < 1$ is a positive number. Then $f_m = 1$.	
\end{lem}
\begin{lem}\label{lem3.1a}{\em\cite{Hu & Li & Yang & 2003}}
Let $f_j\not\equiv0$ $(j=1,2,3)$ be meromorphic functions on such that $f_1$ are not constant, $f_1+f_2+f_3=1$, and such that \beas \sum_{j=1}^{3}\left\{N_{2}\left(r,\frac{1}{f_j}\right)+2\ol N(r,f_j)\right\}< \lambda T(r,f_j)+O(\log^{+}T(r,f_j))\eeas
holds for all $r$ outside possibly a set with finite logarithmic measure, where $\lambda < 1$ is a positive number. Then, either $f_2 = 1$ or $f_3=1$.	
\end{lem}
\begin{lem}\label{lem3.7}{\em\cite{Yan & Yi & 2003}}
Let $a_0(z), a_1(z),\ldots, a_n(z)$ $(n\geq1)$ be meromorphic functions on and $g_0(z), g_1(z),\ldots,g_n(z)$ are entire functions such that $g_j(z)-g_k(z)$ are not constants for $0\leq j<k\leq n$. If $\sum_{j=0}^{n}a_j(z)e^{g_j(z)}\equiv0$, and $T(r,a_j)=o(T(r))$, where $T(r)=\text{min}_{0\leq j<k\leq n} T(r,e^{g_j-g_k})$ for $j=0,1,\ldots,n$, then $a_j(z)\equiv0$ for each $j=0,1,\ldots,n$.
\end{lem}


\begin{proof}[\textbf{Proof of Theorem $\ref{t1}$}]
Let $(f(z),g(z))$ be a pair of finite order transcendental entire solution of \eqref{e2.1}. Let us first make a transformation \bea\label{key} f(z)=\frac{1}{\sqrt{2}}(u+v)\;\;\text{and}\;\;g(z+c)=\frac{1}{\sqrt{2}}(u-v).\eea
\par Then first equation of \eqref{e2.1} yields \beas (1+w)u^2+(1-w)v^2=e^{g_1(z)},\eeas which can be rewritten as \beas \left(\frac{\sqrt{1+w}u}{e^{g_1(z)/2}}+i\frac{\sqrt{1-w}v}{e^{g_1(z)/2}}\right)\left(\frac{\sqrt{1+w}u}{e^{g_1(z)/2}}-i\frac{\sqrt{1-w}v}{e^{g_1(z)/2}}\right)=1.\eeas
\par In view of the above equation and assumption, it follows that there exists a polynomial $p_1(z)$ in $\mathbb{C}^n$ such that \bea\label{e3.1} \frac{u\sqrt{1+w}}{e^{g_1(z)/2}}+i\frac{v\sqrt{1-w}}{e^{g_1(z)/2}}=e^{p_1(z)}\;\;\text{and}\;\;\frac{u\sqrt{1+w}}{e^{g_1(z)/2}}-i\frac{v\sqrt{1-w}}{e^{g_1(z)/2}}=e^{-p_1(z)}.\eea
\par Let \bea\label{e3.2} \gamma_1(z)=p_1(z)+\frac{1}{2}g_1(z),\;\;\gamma_2(z)=-p_1(z)+\frac{1}{2}g_1(z).\eea
\par Then from \eqref{e3.1} and \eqref{e3.2}, we obtain \bea\label{e3.4} u=\frac{1}{2\sqrt{1+w}}\left[e^{\gamma_1}+e^{\gamma_2}\right],\;\; v=\frac{1}{2i\sqrt{1-w}}\left[e^{\gamma_1}-e^{\gamma_2}\right].\eea
\par Therefore, in view of \eqref{key} and \eqref{e3.4}, we obtain \bea\label{e3.5} f(z)=\frac{1}{\sqrt{2}}\left[A_1e^{\gamma_1}+A_2e^{\gamma_2}\right],\;\;g(z+c)=\frac{1}{\sqrt{2}}\left[A_2e^{\gamma_1}+A_1e^{\gamma_2}\right].\eea
\par Similarly, in view of the second equation of \eqref{e2.1}, there exists a polynomial $p_2(z)$ such that \bea\label{e3.6} g(z)=\frac{1}{\sqrt{2}}\left[A_1e^{\gamma_3}+A_2e^{\gamma_4}\right],\;\;f(z+c)=\frac{1}{\sqrt{2}}\left[A_2e^{\gamma_3}+A_1e^{\gamma_4}\right],\eea where \bea\label{e3.7} \gamma_3(z)=p_2(z)+\frac{1}{2}g_2(z),\;\;\gamma_4(z)=-p_2(z)+\frac{1}{2}g_2(z).\eea 
\par From \eqref{e3.5} and \eqref{e3.6}, we easily get \bea\label{e3.8} \begin{cases}A_1e^{\gamma_1(z+c)-\gamma_4(z)}+A_2e^{\gamma_2(z+c)-\gamma_4(z)}-A_2e^{\gamma_3(z)-\gamma_4(z)}=A_1,\\A_1e^{\gamma_3(z+c)-\gamma_2(z)}+A_2e^{\gamma_4(z+c)-\gamma_2(z)}-A_2e^{\gamma_1(z)-\gamma_1(z)}=A_1.\end{cases}\eea
\par Now we discuss the following four possible cases.\vspace{1mm}
\par \textbf{Case 1:} Let $\gamma_1(z)-\gamma_2(z)=\zeta_1$ and $\gamma_3(z)-\gamma_4(z)=\zeta_2$, where $\zeta_1,\zeta_2$ are two complex constants.\vspace{1.2mm}
\par Then, it clearly follows from \eqref{e3.2} and \eqref{e3.7} that $p_1(z)$ and $p_2(z)$ both are constants in $\mathbb{C}$. Let $e^{p_1}=\xi_1$ and $e^{p_2}=\xi_2$.\vspace{1.2mm}
\par Then, from \eqref{e3.5} and \eqref{e3.6}, it follows that \bea\label{e3.9} \begin{cases}
f(z)=\frac{A_1\xi_1+A_2\xi_1^{-1}}{\sqrt{2}}e^{\frac{1}{2}g_1(z)},\\g(z+c)=\frac{A_2\xi_1+A_1\xi_1^{-1}}{\sqrt{2}}e^{\frac{1}{2}g_1(z)},\\g(z)=\frac{A_1\xi_2+A_2\xi_2^{-1}}{\sqrt{2}}e^{\frac{1}{2}g_2(z)},\\f(z+c)=\frac{A_2\xi_2+A_1\xi_2^{-1}}{\sqrt{2}}e^{\frac{1}{2}g_2(z)}.\end{cases}\eea
\par Since $f,g$ are transcendental entire, it follows from \eqref{e3.9} that $A_1\xi_1+A_2\xi_1^{-1}$, $A_2\xi_1+A_1\xi_1^{-1}$, $A_1\xi_2+A_2\xi_2^{-1}$ and $A_2\xi_2+A_1\xi_2^{-1}$ are all non zero constants. Since $p_1$ and $p_2$ are constants and $f,g$ are transcendental entire, it follows from \eqref{e3.9} that $g_1$ and $g_2$ are non-constants.\vspace{1.2mm}
\par After simple calculations, we have from \eqref{e3.9} that \bea\label{e3.10} \begin{cases}(A_1\xi_1+A_2\xi_1^{-1})e^{\frac{1}{2}[g_1(z+c)-g_2(z)]}=A_2\xi_2+A_1\xi_2^{-1},\\(A_1\xi_2+A_2\xi_2^{-1})e^{\frac{1}{2}[g_2(z+c)-g_1(z)]}=A_2\xi_1+A_1\xi_1^{-1}.\end{cases}\eea 
\par Since $g_1,g_2$ are non-constant polynomials in $\mathbb{C}^n$, from \eqref{e3.10}, we see that $g_1(z+c)-g_2(z)=\eta_1$ and $g_2(z+c)-g_1(z)=\eta_2$, where $\eta_1,\eta_2\in\mathbb{C}$. Otherwise, left hand sides of both the equations are transcendental, whereas right hand sides will be polynomial, a contradiction. Therefore, $g_1(z+2c)-g_2(z)=g_2(z+2c)-g_1(z)=\zeta_1+\zeta_2$. Thus, \beas g_1(z)=L(z)+\Phi(z)+d_1\;\;\text{and}\;\; g_2(z)=L(z)+\Phi(z)+d_2,\eeas where $L(z)=\sum_{j=1}^{n}a_jz_j$ and $\Phi(z)$ is a polynomial in $\mathbb{C}^n$ defined as in \eqref{e2.5}, $a_j,d_1,d_2\in\mathbb{C}$, $j=1,2,\ldots,n$.\vspace{1.2mm}
\par Thus, from \eqref{e3.10}, we get \bea\label{e3.11a} e^{\frac{1}{2}[L(c)+d_1-d_2]}=\frac{A_2\xi_2+A_1\xi_2^{-1}}{A_1\xi_1+A_2\xi_1^{-1}}\;\;\text{and}\;\;e^{\frac{1}{2}[L(c)+d_2-d_1]}=\frac{A_2\xi_1+A_1\xi_1^{-1}}{A_1\xi_2+A_2\xi_2^{-1}},\eea which imply \beas e^{L(c)}=\frac{(A_2\xi_2+A_1\xi_2^{-1})(A_2\xi_1+A_1\xi_1^{-1})}{(A_1\xi_1+A_2\xi_1^{-1})(A_1\xi_2+A_2\xi_2^{-1})}\;\;\text{and}\;\; e^{d_1-d_2}=\frac{(A_2\xi_2+A_1\xi_2^{-1})(A_1\xi_2+A_2\xi_2^{-1})}{(A_1\xi_1+A_2\xi_1^{-1})(A_2\xi_1+A_1\xi_1^{-1})}.\eeas
\par Therefore, from \eqref{e3.9}, we obtain \beas f(z)=\frac{A_1\xi_1+A_2\xi_1^{-1}}{\sqrt{2}}e^{\frac{1}{2}[L(z)+\Phi(z)+d_1]}\;\;\text{and}\;\;g(z)=\frac{A_1\xi_2+A_2\xi_2^{-1}}{\sqrt{2}}e^{\frac{1}{2}[L(z)+\Phi(z)+d_2]}.\eeas
\par \textbf{Case 2.} Let $\gamma_1(z)-\gamma_2(z)$ and $\gamma_3(z)-\gamma_4(z)$ both are non-constants. Then, by Lemma \ref{lem3.1a}, it follows from \eqref{e3.8} that \beas A_1e^{\gamma_1(z+c)-\gamma_4(z)}=A_1\;\;\text{or}\;\;A_2e^{\gamma_2(z+c)-\gamma_4(z)}=A_1\eeas and \beas A_1e^{\gamma_3(z+c)-\gamma_2(z)}=A_1\;\;\text{or}\;\;A_2e^{\gamma_4(z+c)-\gamma_2(z)}=A_1.\eeas
\par Now, we consider the following four possible subcases.\vspace{1.2mm}
\par \textbf{Subcase 2.1.} Let \bea\label{e3.11} A_1e^{\gamma_1(z+c)-\gamma_4(z)}=A_1\;\;\text{and}\;\;A_1e^{\gamma_3(z+c)-\gamma_2(z)}=A_1.\eea
\par Then, from \eqref{e3.8} and \eqref{e3.11}, we obtain \bea\label{e3.12}A_2e^{\gamma_2(z+c)-\gamma_3(z)}=A_2\;\;\text{and}\;\;A_2e^{\gamma_4(z+c)-\gamma_1(z)}=A_2.\eea
\par Since $\gamma_j$'s, $j=1,2,3,4$ are all polynomials, it follows from \eqref{e3.9} and \eqref{e3.10} that $\gamma_1(z+c)-\gamma_4(z)=\zeta_1$, $\gamma_3(z+c)-\gamma_2(z)=\zeta_2$, $\gamma_2(z+c)-\gamma_3(z)=\zeta_3$ and $\gamma_4(z+c)-\gamma_1(z)=\zeta_4$, $\zeta_j\in\mathbb{C}$ for $j=1,2,3,4$. This implies that 
\beas\gamma_1(z+2c)-\gamma_1(z)=\gamma_4(z+2c)-\gamma_4(z)=\zeta_1+\zeta_2\;\;\text{and}\eeas \beas\gamma_2(z+2c)-\gamma_2(z)=\gamma_3(z+2c)-\gamma_3(z)=\zeta_3+\zeta_4.\eeas 
\par Therefore, we must get \beas \gamma_1(z)=L_1(z)+\Phi(z)+d_1,\;\; \gamma_4(z)=L_1(z)+\Phi(z)+d_2,\eeas \beas\gamma_2(z)=L_2(z)+\Psi(z)+d_3,\;\;\text{and}\;\; \gamma_3(z)=L_2(z)+\Psi(z)+d_4,\eeas where $L_1(z)=\sum_{j=1}^{n}a_jz_j$, $L_2(z)=\sum_{j=1}^{n}b_jz_j$, $\Phi(z)$ and $\Psi(z)$ are defined in \eqref{e2.5} and \eqref{e2.6}, respectively, $a_{j},,b_j,d_i\in\mathbb{C}$ for $i=1,2$ and $j=1,\ldots,n$. As $\gamma_4(z)-\gamma_3(z)$ and $\gamma_2(z)-\gamma_1(z)$ both are non-constants, we must have $$L_1(z)+\Phi(z)\neq L_2(z)+\Psi(z).$$
\par Therefore, in view of \eqref{e3.2} and \eqref{e3.7}, we have \beas g_1(z)=L_1(z)+L_2(z)+\Phi(z)+\Psi(z)+d_1+d_3\;\;\text{and}\eeas \beas g_2(z)=L_1(z)+L_2(z)+\Phi(z)+\Psi(z)+d_2+d_4.\eeas 
\par Therefore, from \eqref{e3.11} and \eqref{e3.12}, we obtain \bea\label{e3.13} \begin{cases}e^{L_1(c)+d_1-d_2}=1,\;\;e^{L_1(c)+d_2-d_1}=1,\\e^{L_2(c)+d_3-d_4}=1,\;\;e^{L_2(c)+d_4-d_3}=1.\end{cases}\eea
\par After simple computation, we obtain from \eqref{e3.13} that \beas e^{2L_1(c)}=1,\;\;e^{2L_2(c)}=1,\;\;e^{2(d_1-d_2)}=1\;\;\text{and}\;\; e^{2(d_3-d_4)}=1,\eeas which yield \beas e^{L_1(c)}=\pm 1,\;\;e^{d_1-d_2}=\pm 1,\;\;e^{L_2(c)}=\pm 1\;\;\text{and}\;\;e^{d_3-d_4}=\pm 1.\eeas
\par If $e^{L_1(c)}=1$, then from \eqref{e3.13}, we obtain $e^{d_1-d_2}=1$. If $e^{L_1(c)}=-1$, then from \eqref{e3.13}, we obtain $e^{d_1-d_2}=-1$. Similarly, If $e^{L_2(c)}=1$, then from \eqref{e3.13}, it follows that $e^{d_3-d_4}=1$. If $e^{L_2(c)}=-1$, then it follows from \eqref{e3.13} that $e^{d_3-d_4}=-1$.\vspace{1.2mm} 
\par Hence, from \eqref{e3.5} and \eqref{e3.6}, we get \beas f(z)=\frac{1}{\sqrt{2}}\left[A_1e^{L_1(z)+\Phi(z)+d_1}+A_2e^{L_2(z)+\Psi(z)+d_3}\right]\;\;\text{and}\eeas\beas g(z)=\frac{1}{\sqrt{2}}\left[A_1e^{L_2(z)+\Psi(z)+d_4}+A_2e^{L_1(z)+\Phi(z)+d_2}\right].\eeas 
\par \textbf{Subcase 2.2.} Let \bea\label{e3.14}A_1e^{\gamma_1(z+c)-\gamma_4(z)}=A_1\;\;\text{and}\;\; A_2e^{\gamma_4(z+c)-\gamma_2(z)}=A_1.\eea
\par Therefore, from \eqref{e3.2} and \eqref{e3.14}, we obtain \bea\label{e3.15}A_2e^{\gamma_2(z+c)-\gamma_3(z)}=A_2\;\;\text{and}\;\; A_1e^{\gamma_3(z+c)-\gamma_1(z)}=A_2.\eea
\par Since $\gamma_j$'s are polynomials, $j=1,2,3,4$, it follows from \eqref{e3.14} and \eqref{e3.15} that $\gamma_1(z+c)-\gamma_4(z)=\zeta_1$, $\gamma_4(z+c)-\gamma_2(z)=\zeta_2$, $\gamma_2(z+c)-\gamma_3(z)=\zeta_3$ and $\gamma_3(z+c)-\gamma_1(z)=\zeta_4$, $\zeta_j\in\mathbb{C}$ for $j=1,2,3,4$. This implies that $\gamma_1(z+4c)-\gamma_1(z)=\gamma_2(z+4c)-\gamma_2(z)=\sum_{j=1}^{4}\zeta_j$. Hence, we can have $\gamma_1(z)=L(z)+\Phi(z)+k_1$ and $\gamma_2(z)=L(z)+\Phi(z)+k_2$, where $L(z)$ and $\Phi(z)$ are the same as defined in Case 1. But, then we must have $\gamma_2(z)-\gamma_1(z)=k_2-k_1\in\mathbb{C}$, a contradiction.\vspace{1.2mm}
\par \textbf{Subcase 2.3.}  Let \beas A_2e^{\gamma_2(z+c)-\gamma_4(z)}=A_1\;\;\text{and}\;\; A_1e^{\gamma_3(z+c)-\gamma_2(z)}=A_1.\eeas
\par By similar argument as used in \textbf{Subcase 2.2}, we can easily get a contradiction.\vspace{1.2mm}
\par \textbf{Subcase 2.4.}  Let \bea\label{e3.16}A_2e^{\gamma_2(z+c)-\gamma_4(z)}=A_1\;\;\text{and}\;\; A_2e^{\gamma_4(z+c)-\gamma_2(z)}=A_1.\eea
\par In view of \eqref{e3.2} and \eqref{e3.16}, it follows that \bea\label{e3.17}A_1e^{\gamma_1(z+c)-\gamma_3(z)}=A_2\;\;\text{and}\;\; A_1e^{\gamma_3(z+c)-\gamma_1(z)}=A_2.\eea
\par Since $\gamma_1, \gamma_2,\gamma_3,\gamma_4$ are all polynomials, from \eqref{e3.16} and \eqref{e3.17}, we get $\gamma_2(z+c)-\gamma_4(z)=l_1$, $\gamma_4(z+c)-\gamma_2(z)=l_2$, $\gamma_1(z+c)-\gamma_3(z)=l_3$ and $\gamma_3(z+c)-\gamma_1(z)=l_4$, $l_j\in\mathbb{C}$ for $j=1,2,3,4$. This implies that $\gamma_1(z+2c)-\gamma_1(z)=\gamma_3(z+2c)-\gamma_3(z)=l_3+l_4$ and $\gamma_2(z+2c)-\gamma_2(z)=\gamma_4(z+2c)-\gamma_4(z)=l_1+l_2$.\vspace{1.2mm}
\par Therefore, we must have \beas \gamma_1(z)=L_1(z)+\Phi(z)+d_1,\;\; \gamma_3(z)=L_1(z)+\Phi(z)+d_2,\eeas \beas\gamma_2(z)=L_2(z)+\Psi(z)+d_3,\;\;\text{and}\;\; \gamma_4(z)=L_2(z)+\Psi(z)+d_4,\eeas where $L_1(z),L_2(z),\Phi(z),\Psi(z)$ are defined in Subcase 2.1. As $\gamma_2-\gamma_1$ and $\gamma_4-\gamma_3$ are non-constants, we have $L_1(z)+\Phi(z)\neq L_2(z)+\Psi(z)$.\vspace{1.2mm}
\par Therefore, in view of \eqref{e3.2} and \eqref{e3.7}, we obtain \beas g_1(z)=L_1(z)+L_2(z)+\Phi(z)+\Psi(z)+d_1+d_3\;\;\text{and}\eeas \beas g_2(z)=L_1(z)+L_2(z)+\Phi(z)+\Psi(z)+d_2+d_4.\eeas 
\par Hence, from \eqref{e3.16} and \eqref{e3.17}, we have \bea\label{e3.18} \begin{cases}A_2e^{L_2(c)+d_3-d_4}=A_1,\;\;A_2e^{L_2(c)+d_4-d_3}=A_1\\A_1e^{L_1(c)+d_1-d_2}=A_2,\;\;A_1e^{L_1(c)+d_2-d_1}=A_2.\end{cases}\eea
\par From \eqref{e3.18}, it follows that \beas e^{2L_1(c)}=A_2^2/A_1^2,\;\;e^{2L_2(c)}=A_1^2/A_2^2,\;\;e^{2(d_1-d_4)}=1\;\;\text{and}\;\; e^{2(d_3-d_4)}=1,\eeas which yield \beas e^{L_1(c)}=\pm A_2/A_1,\;\;e^{L_2(c)}=\pm A_1/A_2,\;\;e^{d_1-d_4}=\pm1\;\;\text{and}\;\; e^{2(d_3-d_4)}=\pm 1.\eeas
\par If $e^{L_1(c)}=A_2/A_1$, then from \eqref{e3.18}, we get $e^{d_1-d_2}=1$. If $e^{L_1(c)}=-A_2/A_1$, then it follows from \eqref{e3.18} that $e^{d_1-d_2}=-1$. Similarly, If $e^{L_2(c)}=A_1/A_2$, then from \eqref{e3.18}, we get $e^{d_3-d_4}=1$. If $e^{L_2(c)}=-A_1/A_2$, then from \eqref{e3.18}, we get $e^{d_3-d_4}=-1$.\vspace{1.2mm} 
\par Hence, from \eqref{e3.5} and \eqref{e3.6}, we get \beas f(z)=\frac{1}{\sqrt{2}}\left[A_1e^{L_1(z)+\Phi(z)+d_1}+A_2e^{L_2(z)+\Psi(z)+d_3}\right]\;\;\text{and}\eeas \beas g(z)=\frac{1}{\sqrt{2}}\left[A_1e^{L_1(z)+\Phi(z)+d_2}+A_2e^{L_2(z)+\Psi(z)+d_4}\right].\eeas 
\par \textbf{Case 3.} Let $\gamma_3(z)-\gamma_4(z)=\eta\in\mathbb{C}$ and $\gamma_1(z)-\gamma_2(z)$ be non-constant. Then, in view of \eqref{e3.7}, we see that $p_2(z)$ is a constant.\vspace{1.2mm}
\par Since $\gamma_1-\gamma_2$ is non-constant, in view of Lemma \ref{lem3.1a}, it follows from the second equation of \eqref{e3.8} that either $A_1e^{\gamma_3(z+c)-\gamma_2(z)}=A_1$ or $A_2e^{\gamma_4(z+c)-\gamma_2(z)}=A_1$.\vspace{1.2mm}
\par First suppose that $A_1e^{\gamma_3(z+c)-\gamma_2(z)}=A_1$. Then, from the second equation of \eqref{e3.8}, we have $A_2e^{\gamma_4(z+c)-\gamma_1(z)}=A_2$. As all $\gamma_j$, $j=1,2,3,4$ are polynomials, it follows that $\gamma_3(z+c)-\gamma_2(z)=m_1$ and $\gamma_4(z+c)-\gamma_1(z)=m_2$, where $m_1,m_2\in\mathbb{C}$. Therefore, $\gamma_1(z)-\gamma_2(z)=m_1-m_2-\eta$, a constant in $\mathbb{C}$, which contradicts to our assumption.\vspace{1.2mm}
\par Similarly, we can get a contradiction for the case $A_2e^{\gamma_4(z+c)-\gamma_2(z)}=A_1$.\vspace{1.2mm}
\par \textbf{Case 4.} Let $\gamma_2-\gamma_1=\eta_1\in\mathbb{C}$ and $\gamma_4-\gamma_3$ be non-constant. Then, by similar argument as used in \textbf{Subcase 3}, we easily get a contradiction.\end{proof}
\begin{proof}[\textbf{Proof of Theorem $\ref{t2}$}]
Let $(f,g)$ is a pair of finite transcendental entire solutions of the system \eqref{e2.3}.\vspace{1.2mm}
\par Then, in a similar manner as in Theorem \ref{t1}, we obtain that \bea\label{e3.19} \begin{cases}
\dfrac{\partial^k f}{\partial z_1^k}=\displaystyle\frac{1}{\sqrt{2}}\left[A_1e^{\gamma_1(z)}+A_2e^{\gamma_2(z)}\right]\\ g(z+c)=\displaystyle\frac{1}{\sqrt{2}}\left[A_2e^{\gamma_1(z)}+A_1e^{\gamma_2(z)}\right]\\\dfrac{\partial^k g}{\partial z_1^k}=\displaystyle\frac{1}{\sqrt{2}}\left[A_1e^{\gamma_3(z)}+A_2e^{\gamma_4(z)}\right]\\f(z+c)=\displaystyle\frac{1}{\sqrt{2}}\left[A_2e^{\gamma_3(z)}+A_1e^{\gamma_4(z)}\right],\end{cases}\eea where $\gamma_1, \gamma_2,\gamma_3,\gamma_4$ are defined in \eqref{e3.2} and \eqref{e3.7}.\vspace{1.2mm}
\par Differentiating partially $k$ times with respect to $z_1$, second and fourth equations of \eqref{e3.19}, we obtain \bea\label{e3.20}\begin{cases}\dfrac{\partial^k f(z+c)}{\partial z_1^k}=\frac{1}{\sqrt{2}}\left[A_2h_1(z)e^{\gamma_3(z)}+A_1h_2(z)e^{\gamma_4(z)}\right],\vspace{1mm}\\\dfrac{\partial^k g(z+c)}{\partial z_1^k}=\frac{1}{\sqrt{2}}\left[A_2h_3(z)e^{\gamma_1(z)}+A_1h_4(z)e^{\gamma_2(z)}\right],\end{cases}\eea where $h_1,h_2,h_3,h_4$ are given by \beas \begin{cases}h_1(z)=\left(\dfrac{\partial \gamma_3}{\partial z_1} \right)^k+M_k\left(\dfrac{\partial^k \gamma_3}{\partial z_1^k},\ldots,\dfrac{\partial \gamma_3}{\partial z_1}\right),\\ h_2(z)=\left(\dfrac{\partial \gamma_4}{\partial z_1} \right)^k+N_k\left(\dfrac{\partial^k \gamma_4}{\partial z_1^k},\ldots,\dfrac{\partial \gamma_4}{\partial z_1}\right)\\h_3(z)=\left(\dfrac{\partial \gamma_1}{\partial z_1} \right)^k+O_k\left(\dfrac{\partial^k \gamma_1}{\partial z_1^k},\ldots,\dfrac{\partial \gamma_1}{\partial z_1}\right),\\h_4(z)=\left(\dfrac{\partial \gamma_2}{\partial z_1} \right)^k+R_k\left(\dfrac{\partial^k \gamma_2}{\partial z_1^k},\ldots,\dfrac{\partial \gamma_2}{\partial z_1}\right),\end{cases}\eeas where $M_k$ is the partial differential polynomial in $\frac{\partial \gamma_3}{\partial z_1}$ of order less than $k$ in which  $\frac{\partial \gamma_3}{\partial z_1}$ appears in product with at least one more partial derivative of $\gamma_3(z)$ of order greater than or equal to $2$. Similar definitions for $N_k, O_k$ and $R_k$.\vspace{1.2mm}		
\par Now, from the first and third equations of \eqref{e3.19} and \eqref{e3.20}, we obtain \bea\label{e3.21}\begin{cases}
A_1e^{\gamma_1(z+c)-\gamma_4(z)}+A_2e^{\gamma_2(z+c)-\gamma_4(z)}-A_2h_1(z)e^{\gamma_3(z)-\gamma_4(z)}=A_1h_2(z),\\A_1e^{\gamma_3(z+c)-\gamma_2(z)}+A_2e^{\gamma_4(z+c)-\gamma_2(z)}-A_2h_3(z)e^{\gamma_1(z)-\gamma_2(z)}=A_1h_4(z).\end{cases}\eea 	
\par Now, we consider the following four possible cases.\vspace{1.2mm}
\par \textbf{Case 1.} Let $\gamma_2(z)-\gamma_1(z)=\eta_1$ and $\gamma_4(z)-\gamma_3(z)=\eta_2$, where $\eta_1,\eta_2\in\mathbb{C}$. Then, in view of \eqref{e3.2} and \eqref{e3.7}, it follows that $h_1(z)$ and $h_2(z)$ are both constants. Set $e^{h_1}=\xi_1$ and $e^{h_2}=\xi_2$, $\xi_1,\xi_2$ are non-zero constants in $\mathbb{C}$. Therefore, in view of \eqref{e3.19}, we obtain \bea\label{e3.22} \begin{cases}
\dfrac{\partial^k f}{\partial z_1^k}=\frac{A_1\xi_1+A_2\xi_1^{-1}}{\sqrt{2}}e^{\frac{1}{2}g_1(z)},\vspace{.5mm}\\g(z+c)=\frac{A_2\xi_1+A_1\xi_1^{-1}}{\sqrt{2}}e^{\frac{1}{2}g_1(z)},\vspace{.5mm}\\\dfrac{\partial^k g}{\partial z_1^k}=\frac{A_1\xi_2+A_2\xi_2^{-1}}{\sqrt{2}}e^{\frac{1}{2}g_2(z)},\vspace{.5mm}\\f(z+c)=\frac{A_2\xi_2+A_1\xi_2^{-1}}{\sqrt{2}}e^{\frac{1}{2}g_2(z)}.\end{cases}\eea
\par As $f,g$ are transcendental entire, in view of \eqref{e3.22}, we see that $g_1(z), g_2(z)$ both are non-constant polynomials and $A_1\xi_1+A_2\xi_1^{-1}$, $A_2\xi_1+A_1\xi_1^{-1}$, $A_1\xi_2+A_2\xi_2^{-1}$, $A_2\xi_2+A_1\xi_2^{-1}$ are all non-zero constants.\vspace{1.2mm}
\par Also, after simple computations, we obtain from \eqref{e3.22} that \bea\label{e3.23} \begin{cases}		e^{\frac{1}{2}[g_1(z+c)-g_2(z)]}=\dfrac{A_2\xi_2+A_1\xi_2^{-1}}{A_1\xi_1+A_2\xi_1^{-1}}r_1(z),\\e^{\frac{1}{2}[g_2(z+c)-g_1(z)]}=\dfrac{A_2\xi_1+A_1\xi_1^{-1}}{A_1\xi_2+A_2\xi_2^{-1}}r_2(z)\end{cases}\eea with \beas r_1(z)=\left(\frac{1}{2}\dfrac{\partial g_2(z)}{\partial z_1}\right)^k+G_1\left(\dfrac{\partial^k g_2(z)}{\partial z_1^k},\ldots,\dfrac{\partial g_2(z)}{\partial z_1}\right)\;\;\text{and}\eeas \beas r_2(z)=\left(\frac{1}{2}\dfrac{\partial g_1(z)}{\partial z_1}\right)^k+G_2\left(\dfrac{\partial^k g_1(z)}{\partial z_1^k},\ldots,\dfrac{\partial g_1(z)}{\partial z_1}\right),\eeas where $G_1$ is a differential polynomial in $g_2(z)$ of order less than $k$ in which $g_2^{\prime}(z)$ appears in the product with at least one more partial derivative of order greater than or equal to $2$, and similar definition for $G_2$.\vspace{1.2mm}
\par Since $g_1(z),g_2(z)$ are non-constant polynomials, from \eqref{e3.23}, conclude that $r_1,r_2, g_1(z+c)-g_2(z)$ and $g_2(z+c)-g_1(z)$ are all constants. Let $g_1(z+c)-g_2(z)=\eta_3$ and $g_2(z+c)-g_1(z)=\eta_4$, $\eta_3,\eta_4\in\mathbb{C}$. This implies that $g_1(z+2c)-g_1(z)=g_2(z+2c)-g_2(z)=\eta_3+\eta_4$. Hence, we can get \beas g_1(z_1,z_2)=\mathcal{L}(z)+\chi(z_2)+d_1,\;\;g_2(z_1,z_2)=\mathcal{L}(z)+\chi(z_2)+d_2,\eeas where $\mathcal{L}(z)=a_1z_1+a_2z_2$, $\chi(z_2)$ is a polynomial in $z_2$ such that $\chi(z_2+c_2)=\chi(z_2)$, $a_1,a_2,d_1,d_2\in\mathbb{C}$. Since, $r_1,r_2$ both are constants, it follows that $$r_1=r_2=\left(\frac{a_1}{2}\right)^k.$$\vspace{1.2mm}
\par Therefore, from \eqref{e3.23}, we see that \bea\label{e3.24} \begin{cases}		e^{\frac{1}{2}[\mathcal{L}(c)+d_1-d_2]}=\left(\dfrac{a_1}{2}\right)^k\dfrac{A_2\xi_2+A_1\xi_2^{-1}}{A_1\xi_1+A_2\xi_1^{-1}},\vspace{1mm}\\e^{\frac{1}{2}[\mathcal{L}(c)+d_2-d_1]}=\left(\dfrac{a_1}{2}\right)^k\dfrac{A_2\xi_1+A_1\xi_1^{-1}}{A_1\xi_2+A_2\xi_2^{-1}}.\end{cases}\eea
\par Hence, from \eqref{e3.24}, it follows that \beas e^{\mathcal{L}(c)}=\left(\dfrac{a_1}{2}\right)^{2k}\dfrac{(A_2\xi_2+A_1\xi_2^{-1})(A_2\xi_1+A_1\xi_1^{-1})}{(A_1\xi_1+A_2\xi_1^{-1})(A_1\xi_2+A_2\xi_2^{-1})}.\eeas
\par Thus, from \eqref{e3.22}, we get \beas	f(z_1,z_2)=\dfrac{A_2\xi_2+A_1\xi_2^{-1}}{\sqrt{2}}e^{\frac{1}{2}[\mathcal{L}(z)+\chi(z_2)-\mathcal{L}(c)+d_2]}\;\;\text{and}\eeas \beas g(z_1,z_2)=\dfrac{A_2\xi_1+A_1\xi_1^{-1}}{\sqrt{2}}e^{\frac{1}{2}[\mathcal{L}(z)+\chi(z_2)-\mathcal{L}(c)+d_1]}.\eeas
\par \textbf{Case 2.} Let $\gamma_1(z)-\gamma_2(z)$ and $\gamma_3(z)-\gamma_4(z)$ both are non-constant.\vspace{1.2mm}
\par Observe that $h_1(z)$ and $h_2(z)$ both can not be simultaneously zero. Otherwise, from the first equation of \eqref{e3.21}, we obtain $e^{\gamma_1(z+c)-\gamma_2(z+c)}=-1$, which is a contradiction as $\gamma_2(z)-\gamma_1(z)$ is non-constant.\vspace{1.2mm}
\par Now, let $h_1(z)\not\equiv0$ and $h_2(z)\equiv0$. Then, first equation of \eqref{e3.21} reduces to \bea\label{e3.25} A_1e^{\gamma_1(z+c)-\gamma_3(z)}+A_2e^{\gamma_2(z+c)-\gamma_3(z)}=A_2h_1(z).\eea
\par As $\gamma_2(z)-\gamma_1(z)$ is a non-constant polynomial, in view of \eqref{e3.25}, it follows that $\gamma_1(z+c)-\gamma_3(z)$ and $\gamma_2(z+c)-\gamma_3(z)$ both are non-constants.\vspace{1.2mm}
\par Rewrite \eqref{e3.25} as \bea\label{e3.26} A_1e^{\gamma_1(z+c)}+A_2e^{\gamma_2(z+c)}-A_2h_1e^{\gamma_3(z)}\equiv0.\eea
\par Therefore, in view of Lemma \ref{lem3.7}, we can easily get a contradiction from \eqref{e3.26}.\vspace{1.2mm}
\par In a similar manner, we can get a contradiction when  $h_1(z)\equiv0$ and $h_2(z)\not\equiv0$. Hence, $h_1(z)\not\equiv0$ and $h_2(z)\not\equiv0$. By similar arguments, we also get $h_3(z)\not\equiv0$ and $h_4(z)\not\equiv0$.\vspace{1.2mm}
\par Therefore, by \eqref{e3.21} and Lemma \ref{lem3.1a}, we obtain \beas \text{either}\;\; A_1e^{\gamma_1(z+c)-\gamma_4(z)}=A_1h_2(z),\;\; \text{or}\;\;A_2e^{\gamma_2(z+c)-\gamma_4(z)}=A_1h_2(z)\eeas and \beas \text{either}\;\; A_1e^{\gamma_3(z+c)-\gamma_2(z)}=A_1h_4(z),\;\; \text{or}\;\;A_2e^{\gamma_4(z+c)-\gamma_2(z)}=A_1h_4(z).\eeas
\par Now, we discuss the four possible subcases.\vspace{1.2mm}
\par \textbf{Subcase 2.1.} Let \bea\label{e3.27}	A_1e^{\gamma_1(z+c)-\gamma_4(z)}=A_1h_2(z)\;\;\text{and}\;\;A_1e^{\gamma_3(z+c)-\gamma_2(z)}=A_1h_4(z).\eea
Then, by \eqref{e3.21} and \eqref{e3.27}, we get 
\bea\label{e3.28}A_2e^{\gamma_2(z+c)-\gamma_3(z)}=A_2h_1(z)\;\;\text{and}\;\;A_2e^{\gamma_4(z+c)-\gamma_1(z)}=A_2h_3(z).\eea
\par Since $\gamma_j$'s are polynomials, $j=1,\ldots,4$, \eqref{e3.27} and \eqref{e3.28} yield that $h_1,h_2,h_3,h_4$, $\gamma_1(z+c)-\gamma_4(z),\gamma_3(z+c)-\gamma_2(z),\gamma_2(z+c)-\gamma_3(z),\gamma_4(z+c)-\gamma_1(z)$ are all constants. Let $\gamma_1(z+c)-\gamma_4(z)=k_1$, $\gamma_3(z+c)-\gamma_2(z)=k_2$, $\gamma_2(z+c)-\gamma_3(z)=k_3$ and $\gamma_4(z+c)-\gamma_1(z)=k_4$, where $k_j\in\mathbb{C}$, $j=1,\ldots,4$. These imply $\gamma_1(z+2c)-\gamma_1(z)=\gamma_4(z+2c)-\gamma_4(z)=k_1+k_4$ and $\gamma_2(z+2c)-\gamma_2(z)=\gamma_3(z+2c)-\gamma_3(z)=k_2+k_3$. Thus, we have \beas \gamma_1(z)=L_1(z)+\Phi_1(z_2)+d_1,\;\; \gamma_4(z)=L_1(z)+\Phi_1(z_2)+d_2,\eeas \beas\gamma_2(z)=L_2(z)+\Psi_1(z_2)+d_3,\;\;\text{and}\;\; \gamma_3(z)=L_2(z)+\Psi_1(z_2)+d_4,\eeas where $L_1(z)=a_1z_1+a_2z_2$, $L_2(z)=b_1z_1+b_2z_2$, $\Phi_1(z_2)$ and $\Psi_1(z_2)$ are polynomials in $z_2$ such that $\Phi_1(z_2+c_2)=\Phi_1(z_2)$ and $\Psi_1(z_2+c_2)=\Psi_1(z_2)$, $a_{j},,b_j,e_j,t_j\in\mathbb{C}$ for $j=1,2$.\vspace{1.2mm}
\par Since $\gamma_1(z)-\gamma_2(z)$ and $\gamma_3(z)-\gamma_4(z)$ are non-constants, we conclude that $$L_1(z)+\Phi_1(z_2)\neq L_2(z)+\Psi_1(z_2).$$
\par Therefore, in view of \eqref{e3.2} and \eqref{e3.7}, we have \bea\label{1}\begin{cases}
g_1(z_1,z_2)=L_1(z)+L_2(z)+\Phi_1(z_2)+\Psi_1(z_2)+d_1+d_3,\\g_2(z_1,z_2)=L_1(z)+L_2(z)+\Phi_1(z_2)+\Psi_1(z_2)+d_2+d_4.\end{cases}\eea \vspace{1.2mm}
\par Hence, from \eqref{e3.27} and \eqref{e3.28}, we obtain \bea\label{e3.29}\begin{cases} e^{L_1(c)+d_1-d_2}=a_{1}^k,\;\;e^{L_1(c)+d_2-d_1}=a_{1}^k,\\e^{L_2(c)+d_3-d_4}=b_{1}^k,\;\;e^{L_2(c)+d_4-d_3}=b_{1}^k.\end{cases}\eea
\par From \eqref{e3.29}, we obtain \beas e^{2L_1(c)}=a_{1}^{2k},\;\;e^{2(d_1-d_2)}=1,\;\;e^{2L_2(c)}=b_{1}^{2k}\;\;\text{and}\;\;e^{2(d_3-d_4)}=1.\eeas
\par Therefore, we have $e^{L_1(c)}=\pm a_1^k,e^{L_2(c)}=\pm b_1^k,$ $e^{d_1-d_2}=\pm 1$ and $e^{d_3-d_4}=\pm 1$.\vspace{1.2mm}
\par If $e^{L_1(c)}=a_1^k$, then from \eqref{e3.29}, we get $e^{d_1-d_2}=1$. If $e^{L_1(c)}=-a_1^k$, then from \eqref{e3.29}, we get $e^{d_1-d_2}=-1$. If $e^{L_2(c)}=b_1^k$, then from \eqref{e3.29}, we get $e^{d_3-d_4}=1$ and if $e^{L_2(c)}=-b_1^k$, then we get $e^{d_3-d_4}=-1$.
\par Thus, from fourth and second equations of \eqref{e3.19}, we obtain 
\beas\begin{cases}		f(z)=\frac{1}{\sqrt{2}}\left[A_2e^{L_2(z)+\Psi_1(z_2)-L_2(c)+d_4}+A_1e^{L_1(z)+\Phi_1(z_2)-L_1(c)+d_2}\right]\vspace{1mm}\\g(z)=\frac{1}{\sqrt{2}}\left[A_2e^{L_1(z)+\Phi_1(z_2)-L_1(c)+d_1}+A_1e^{L_2(z)+\Psi_1(z_2)-L_2(c)+d_3}\right].\end{cases} \eeas
\par \textbf{Subcase 2.2.} Let \bea\label{e3.30}	A_1e^{\gamma_1(z+c)-\gamma_4(z)}=A_1h_2(z)\;\;\text{and}\;\;A_2e^{\gamma_4(z+c)-\gamma_2(z)}=A_1h_4(z).\eea
Then, in view of \eqref{e3.21} and \eqref{e3.30}, we obtain \bea\label{e3.31}A_2e^{\gamma_2(z+c)-\gamma_3(z)}=A_2h_1(z)\;\;\text{and}\;\;A_1e^{\gamma_3(z+c)-\gamma_1(z)}=A_2h_3(z).\eea
\par Since $\gamma_1(z),\gamma_2(z),\gamma_3(z), \gamma_4(z)$ are polynomials, we conclude from \eqref{e3.30} and \eqref{e3.31} that $\gamma_1(z+c)-\gamma_3(z)=m_1$, $\gamma_4(z+c)-\gamma_1(z)=m_2$, $\gamma_2(z+c)-\gamma_4(z)=m_3$ and $\gamma_3(z+c)-\gamma_2(z)=m_4$, where $m_j\in\mathbb{C}$, $j=1,\ldots,4$. These imply $\gamma_1(z+2c)-\gamma_1(z)=\gamma_2(z+4c)-\gamma_2(z)=\sum_{j=1}^{4}m_j$. Therefore, $\gamma_1(z)=L_1(z)+\Phi_1(s_1)+d_1$ and $\gamma_2(z)=L_1(z)+\Phi_1(s_1)+d_2$, where $L_1(z)$ and $\Phi_1(s_1)$ are defined in Subcase 2.1, $d_1,d_2\in\mathbb{C}$. Hence, $\gamma_1(z)-\gamma_2(z)=d_1-d_2\in\mathbb{C}$, a contradiction.\vspace{1.2mm}
\par \textbf{Subcase 2.3.} Let \beas A_2e^{\gamma_2(z+c)-\gamma_4(z)}=A_1h_2(z)\;\;\text{and}\;\;A_1e^{\gamma_3(z+c)-\gamma_2(z)}=A_1h_4(z).\eeas
\par Then, by similar arguments as in Subcase 2.2, we easily get a contradiction.\vspace{1.2mm}
\textbf{Subcase 2.4.} Let \bea\label{e3.32}	A_2e^{\gamma_2(z+c)-\gamma_4(z)}=A_1h_2(z)\;\;\text{and}\;\;A_2e^{\gamma_4(z+c)-\gamma_2(z)}=A_1h_4(z).\eea
\par Then, by \eqref{e3.21} and \eqref{e3.32}, we get 
\bea\label{e3.33}A_1e^{\gamma_1(z+c)-\gamma_3(z)}=A_2h_1(z)\;\;\text{and}\;\;A_1e^{\gamma_3(z+c)-\gamma_1(z)}=A_2h_3(z).\eea
\par As $\gamma_1(z),\gamma_2(z),\gamma_3(z), \gamma_4(z)$ are polynomials, from \eqref{e3.32} and \eqref{e3.33}, we conclude that $h_1,h_2,h_3,h_4$, $\gamma_2(z+c)-\gamma_4(z),\gamma_4(z+c)-\gamma_2(z),\gamma_1(z+c)-\gamma_3(z),\gamma_3(z+c)-\gamma_1(z)$ are all constants. Let $\gamma_2(z+c)-\gamma_4(z)=n_1$, $\gamma_4(z+c)-\gamma_2(z)=n_2$, $\gamma_1(z+c)-\gamma_3(z)=n_3$ and $\gamma_3(z+c)-\gamma_1(z)=n_4$, where $n_j\in\mathbb{C}$, $j=1,\ldots,4$. Thus, we get $\gamma_1(z+2c)-\gamma_1(z)=\gamma_3(z+2c)-\gamma_3(z)=n_1+n_2$ and $\gamma_2(z+2c)-\gamma_2(z)=\gamma_4(z+2c)-\gamma_4(z)=n_3+n_4$. Hence, by similar arguments as in Subcase 2.1, we obtain 
\beas \gamma_1(z)=L_1(z)+\Phi_1(z_2)+d_1,\;\; \gamma_3(z)=L_1(z)+\Phi_1(z_2)+d_2,\eeas \beas\gamma_2(z)=L_2(z)+\Psi_1(z_2)+d_3,\;\;\text{and}\;\; \gamma_4(z)=L_2(z)+\Psi_1(z_2)+d_4,\eeas where $L_(z),L_2(z),\Phi_1(z_2),\Psi_1(z_2)$ are same as in Subcase 2.1 with $L_1(z)+\Phi_1(z_2)\neq L_2(z)+\Psi_1(z_2)$. Thus, in view of \eqref{e3.2} and \eqref{e3.7}, we obtain the same form of $g_1(z_1,z_2)$ and $g_2(z_1,z_2)$ as in \eqref{1}.\vspace{1.2mm}
\par Therefore, in view of \eqref{e3.32} and \eqref{e3.33}, we have \bea\label{e3.34}\begin{cases} A_2e^{L_2(c)+d_3-d_4}=A_1b_{1}^k,\;\;A_2e^{L_2(c)+d_4-d_3}=A_1b_{1}^k,\\A_1e^{L_1(c)+d_1-d_2}=A_2a_{1}^k,\;\;A_1e^{L_1(c)+d_2-d_1}=A_2a_{1}^k.\end{cases}\eea
\par From \eqref{e3.34}, we obtain \beas e^{2L_1(c)}=\frac{A_2^2}{A_1^2}a_{1}^{2k}\,\;\;e^{2(d_1-d_2)}=1,\;\;e^{2L_2(c)}=\frac{A_1^2}{A_2^2}b_{1}^{2k}\;\;\text{and}\;\;e^{2(d_3-d_4)}=1,\eeas which imply \beas e^{L_1(c)}=\pm\frac{A_2}{A_1}a_{1}^{k}\,\;\;e^{d_1-d_2}=\pm1,\;\;e^{L_2(c)}=\pm\frac{A_1}{A_2}b_{1}^{k}\;\;\text{and}\;\;e^{d_3-d_4}=\pm1.\eeas
\par If $e^{L_1(c)}=A_2a_1^k/A_1$, then from \eqref{e3.34}, we get $e^{d_1-d_2}=1$. If $e^{L_1(c)}=-A_2a_1^k/A_1$, then we have $e^{d_1-d_2}=-1$. If $e^{L_2(c)}=A_1b_1^k/A_2$, then from \eqref{e3.34}, we get $e^{d_3-d_4}=1$. If $e^{L_2(c)}=-A_1b_1^k/A_2$, then we have $e^{d_3-d_4}=-1$.\vspace{1.2mm}
\par Hence, from fourth and second equations of \eqref{e3.19}, we get \beas\begin{cases}
f(z)=\frac{1}{\sqrt{2}}\left[A_2e^{L_1(z)+\Phi_1(z_2)-L_1(c)+d_2}+A_1e^{L_2(z)+\Psi_1(z_2)-L_2(c)+d_4}\right],\vspace{1mm}\\g(z)=\frac{1}{\sqrt{2}}\left[A_2e^{L_1(z)+\Phi_1(z_2)-L_1(c)+d_1}+A_1e^{L_2(z)+\Psi_1(z_2)-L_2(c)+d_3}\right].
\end{cases} \eeas 
\par \textbf{Case 3.} Let $\gamma_3(z)-\gamma_4(z)=\eta\in\mathbb{C}$ and $\gamma_1(z)-\gamma_2(z)$ be non-constant. Then the first equation of \eqref{e3.21} reduces to \bea\label{e3.35}A_1e^{\gamma_1(z+c)-\gamma_4(z)}+A_2e^{\gamma_2(z+c)-\gamma_4(z)}=A_1h_2(z)+A_2h_1(z)e^{\eta}.\eea
\par Note that $A_1h_2(z)+A_2h_1(z)e^{\eta}\not\equiv0$. Otherwise from \eqref{e3.35}, we have \beas A_1e^{\gamma_1(z+c)-\gamma_2(z+c)}=-A_2,\eeas which implies that $\gamma_1(z+c)-\gamma_2(z+c)$ and hence $\gamma_1(z)-\gamma_2(z)$ becomes constant. This is a contradiction.\vspace{1.2mm} 
\par Now, in view of \eqref{e3.35}, it follows that \bea\label{e3.36} T\left(r,e^{\gamma_1(z+c)-\gamma_4(z)}\right)=T\left(r,e^{\gamma_2(z+c)-\gamma_4(z)}\right)+O(1).\eea
\par We note that \beas N\left(r,e^{\gamma_1(z+c)-\gamma_4(z)}\right)=N\left(r,\frac{1}{e^{\gamma_1(z+c)-\gamma_4(z)}}\right)=S\left(r,e^{\gamma_1(z+c)-\gamma_4(z)}\right)\eeas and \beas N\left(r,\frac{1}{e^{\gamma_1(z+c)-\gamma_4(z)}-w}\right)=N\left(r,\frac{1}{e^{\gamma_2(z+c)-\gamma_4(z)}}\right)=S\left(r,e^{\gamma_2(z+c)-\gamma_4(z)}\right),\eeas where $w=h_2+A_2h_1e^{\eta}/A_1$.
\par Now by the second fundamental theorem of Nevanlinna for several complex variables, we get \beas && T\left(r,e^{\gamma_1(z+c)-\gamma_4(z)}\right)\\&& \leq N\left(r,e^{\gamma_1(z+c)-\gamma_4(z)}\right)+N\left(r,\frac{1}{e^{\gamma_1(z+c)-\gamma_4(z)}}\right)+N\left(r,\frac{1}{e^{\gamma_1(z+c)-\gamma_4(z)}-w}\right)\\&&+S\left(r,e^{\gamma_1(z+c)-\gamma_4(z)}\right)\\&&\leq S\left(r,e^{\gamma_1(z+c)-\gamma_4(z)}\right)+S\left(r,e^{\gamma_2(z+c)-\gamma_4(z)}\right).\eeas
\par This implies that $\gamma_1(z+c)-\gamma_4(z)$ is constant. In view of \eqref{e3.36}, we see that $\gamma_2(z+c)-\gamma_4(z)$ is also constant. These imply $\gamma_1(z+c)-\gamma_2(z+c)$ and hence $\gamma_1(z)-\gamma_2(z)$ is a constant. This is a contradiction.\vspace{1.2mm}
\par \textbf{Case 4.} Let $\gamma_1-\gamma_2=\eta_1\in\mathbb{C}$ and $\gamma_3-\gamma_4$ be non-constant.  Then, by similar arguments as used in Subcase 3, we can get a contradiction. So, we omit the details.\end{proof}
\begin{proof}[\textbf{Proof of Theorem $\ref{t3}$}]
Let $(f(z),g(z))$ be a pair of finite order transcendental entire solution of \eqref{e1.6}.\vspace{1.2mm}
\par Then, by similar argument as in Theorem \ref{t1}, we obtain \bea\label{e3.37} \begin{cases}		\dfrac{\partial^k f}{\partial z_1^k}=\dfrac{1}{\sqrt{2}}\left[A_1e^{p_1(z)}+A_2e^{-p_1(z)}\right],\\ g(z+c)-g(z)=\displaystyle\frac{1}{\sqrt{2}}\left[A_2e^{p_1(z)}+A_1e^{-p_1(z)}\right],\\\dfrac{\partial^k g}{\partial z_1^k}=\displaystyle\frac{1}{\sqrt{2}}\left[A_1e^{p_2(z)}+A_2e^{-p_2(z)}\right],\\f(z+c)-f(z)=\displaystyle\frac{1}{\sqrt{2}}\left[A_2e^{p_2(z)}+A_1e^{-p_2(z)}\right],\end{cases}\eea where $p_1(z)$ and $p_2(z)$ are two non-constant polynomials.\vspace{1.2mm}
\par Differentiating fourth equation of \eqref{e3.37} $k$ times, we get \bea\label{e3.38} \dfrac{\partial^k f(z+c)}{\partial z_1^k}-\dfrac{\partial^k f}{\partial z_1^k}=\frac{1}{\sqrt{2}}\left[A_2H_1e^{p_2(z)}+A_1H_2e^{-p_2(z)}\right],\eea where $H_1=\left(\frac{\partial p_2}{\partial z_1}\right)^k+M_k\left(\frac{\partial^k p_2}{\partial z_1^k},\ldots,\frac{\partial p_2}{\partial z_1}\right)$ and $H_2=\left(-\frac{\partial p_2}{\partial z_1}\right)^k+N_k\left(\frac{\partial^k p_2}{\partial z_1^k},\ldots,\frac{\partial p_2}{\partial z_1}\right)$, where $M_k$ is a partial differential polynomial in $p_2(z)$ of degree less than $k$ in which $\frac{\partial p_2}{\partial z_1}$ appears in the product with at least one more partial derivative of $p_2$ of order greater than $1$. Similar definition is for $N_k$.\vspace{1.2mm}
\par Now in view of the first equation of \eqref{e3.37} and \eqref{e3.38}, we obtain 
\bea&&\label{e3.39}	A_1e^{p_1(z+c)+p_2(z)}+A_2e^{-p_1(z+c)+p_2(z)}-A_1e^{p_1(z)+p_2(z)}-A_2e^{-p_1(z)+p_2(z)}\nonumber\\&&-A_2H_1e^{2p_2(z)}=A_1H_2.\eea
\par Similarly in view of the second and third equations of \eqref{e3.37}, we obtain \bea&&\label{e3.40}A_1e^{p_2(z+c)+p_1(z)}+A_2e^{-p_2(z+c)+p_1(z)}-A_1e^{p_2(z)+p_1(z)}-A_2e^{-p_2(z)+p_1(z)}\nonumber\\&&-A_2H_3e^{2p_1(z)}=A_1H_4,\eea where $H_3=\left(\frac{\partial p_1}{\partial z_1}\right)^k+O_k\left(\frac{\partial^k p_1}{\partial z_1^k},\ldots,\frac{\partial p_1}{\partial z_1}\right)$ and $H_4=\left(-\frac{\partial p_1}{\partial z_1}\right)^k+R_k\left(\frac{\partial^k p_1}{\partial z_1^k},\ldots,\frac{\partial p_1}{\partial z_1}\right)$, where $O_k$ is a differential polynomial in $p_1(z)$ of degree less than $k$ in which $\frac{\partial p_1}{\partial z_1}$ appears in the product with at least one higher order derivative of $p_1$. Similar definition is for $R_k$.\vspace{1.2mm}
\par Now, we consider the following two possible cases.\vspace{1.2mm}
\par \textbf{Case 1.} Let $p_2(z)-p_1(z)=\eta$, $\eta\in\mathbb{C}$. Then \eqref{e3.39} and \eqref{e3.40} yield \bea\label{e3.41}\begin{cases}A_1e^{\eta}e^{p_1(z+c)+p_1(z)}+A_2e^{\eta}e^{p_1(z)-p_1(z+c)}-e^{\eta}(A_1+A_2H_1e^{\eta})e^{2p_1(z)}\\=A_1H_2+A_2e^{\eta},\\A_1e^{\eta}e^{p_1(z+c)+p_1(z)}+A_2e^{-\eta}e^{p_1(z)-p_1(z+c)}-(A_1e^{\eta}+A_2H_3)e^{2p_1(z)}\\=A_1H_4+A_2e^{-\eta}.\end{cases}\eea
\par Note that $A_1+A_2H_1e^{\eta}$ and $A_1H_2+A_2e^{\eta}$ can not be zero simultaneously. Otherwise, from the first equation of \eqref{e3.41}, we get $p_1(z+c)$, and hence $p_1(z)$ is constant, a contradiction.\vspace{1.2mm}
\par Next, suppose that $A_1+A_2H_1e^{\eta}\not\equiv0$ and $A_1H_2+A_2e^{\eta}\equiv0$. The the first equation of \eqref{e3.41} reduces to \bea\label{e3.42}A_1e^{p_1(z+c)-p_1(z)}+A_2e^{-(p_1(z)+p_1(z+c))}=A_1+A_2H_1e^{\eta}.\eea
\par Now, observe that \beas N\left(r,e^{-(p_1(z+c)+p_1(z))}\right)=N\left(r,\frac{1}{e^{-(p_1(z+c)+p_1(z))}}\right)=S\left(r,e^{-(p_1(z+c)+p_1(z))}\right)\eeas and in view of \eqref{e3.42}, we have \beas N\left(r,\frac{1}{e^{-(p_1(z+c)+p_1(z))}-w}\right)=N\left(r,\frac{1}{e^{p_1(z+c)-p_1(z)}}\right)=S\left(r,e^{-(p_1(z+c)+p_1(z))}\right),\eeas where $w=(A_1+A_2H_1e^{\eta})/A_2$.\vspace{1.2mm}
\par By the second fundamental theorem of Nevanlinna, we obtain \beas&& T\left(r,e^{-(p_1(z+c)+p_1(z))}\right)\\&&\leq \ol N\left(r,e^{-(p_1(z+c)+p_1(z))}\right)+\ol N\left(r,\frac{1}{e^{-(p_1(z+c)+p_1(z))}}\right)\\&&+ \ol N\left(r,\frac{1}{e^{-(p_1(z+c)+p_1(z))}-w}\right)+S\left(r,e^{-(p_1(z+c)+p_1(z))}\right)\\&&\leq S\left(r,e^{-(p_1(z+c)+p_1(z))}\right)+S\left(r,e^{p_1(z+c)-p_1(z)}\right).\eeas
\par This implies that $p_1(z+c)+p_1(z)$ is constant, and hence $p_1(z)$ is constant, which is a contradiction.\vspace{1.2mm}
\par Similarly, we can get a contradiction for the case $A_1+A_2H_1e^{\eta}\equiv0$ and $A_1H_2+A_2e^{\eta}\not\equiv0$. Hence, $A_1+A_2H_1e^{\eta}\not\equiv0$ and $A_1H_2+A_2e^{\eta}\not\equiv0$.\vspace{1.2mm}
\par In a similar manner, we can show that $A_1e^{\eta}+A_2H_3\not\equiv0$ and $A_1H_4+A_2e^{-\eta}\not\equiv0$.\vspace{1.2mm}
\par Therefore, in view of Lemma \ref{lem3.1} and \eqref{e3.41}, we obtain \bea\label{e3.43}\begin{cases}
A_2e^{\eta}e^{p_1(z)-p_1(z+c)}=A_1H_2+A_2e^{\eta},\\A_2e^{-\eta}e^{p_1(z)-p_1(z+c)}=A_1H_4+A_2e^{-\eta}.\end{cases}\eea 
\par By \eqref{e3.41} and \eqref{e3.43}, we get \bea\label{e3.44}\begin{cases}A_1e^{p_1(z+c)-p_1(z)}=A_1+A_2H_1e^{\eta},\\A_1e^{\eta}e^{p_1(z+c)-p_1(z)}=A_2H_3+A_1e^{\eta}.\end{cases}\eea
\par Since $p_1(z)$ is a non-constant polynomial, it follows from \eqref{e3.43} and \eqref{e3.44} that $p_1(z+c)-p_1(z)$, $H_1,H_2,H_3,H_4$ are all constants. Otherwise, L.H.S. of both the equations of \eqref{e3.43} and \eqref{e3.44} are transcendental entire, whereas R.H.S's. are polynomials. Therefore, we must have $p_1(z)=L(z)+\Phi_1(z_2)+\gamma$, where $L(z)=\alpha z_1+\beta z_2$, $\Phi_1(z_2)$ is a polynomial in $z_2$ with $\Phi_1(z_2+c_2)=\Phi_1(z_2)$, $\alpha,\beta,\gamma\in\mathbb{C}$. Thus, $p_2(z)=L(z)+\Phi_1(z_2)+\gamma+\eta$. \vspace{1.2mm}
\par Therefore, in view of \eqref{e3.43}, \eqref{e3.44} and the form of $H_1$, $H_2$, $H_3$ and $H_4$, we obtain \bea\label{e3.45} \begin{cases}A_2e^{\eta}e^{-\alpha c}=A_1(-\alpha)^k+A_2e^{\eta},\;\; A_2e^{-\eta}e^{-\alpha c}=A_1(-\alpha)^k+A_2e^{-\eta},\\ A_1e^{\alpha c}=A_1+A_2\alpha^k e^{\eta},\;\; A_1e^{\eta}e^{\alpha c}=A_2\alpha^k+A_1e^{\eta}.\end{cases}\eea
\par From \eqref{e3.45}, we obtain \beas e^{L(c)}=\frac{A_2e^{\eta}}{(-1)^kA_1\alpha^k+A_2e^{\eta}},\;\;e^{2\eta}=1\;\;\text{and}\;\;\alpha^k=-\frac{(-1)^kA_1^2+A_2^2}{(-1)^kA_1A_2e^{\eta}}\eeas\vspace{1.2mm}
\par Thus, from \eqref{e3.37}, we obtain \beas f(z)=\frac{A_1e^{L(z)+\Phi_1(z_2)+\gamma}+A_2e^{-(L(z)+\Phi_1(z_2)+\gamma)}}{\sqrt{2}\alpha^k}\;\; \text{and}\eeas \beas g(z)=\frac{A_1e^{L(z)+\Phi_1(z_2)+\gamma+\eta}+A_2e^{-(L(z)+\Phi_1(z_2)+\gamma+\eta)}}{\sqrt{2}\alpha^k}.\eeas
\par \textbf{Case 2.} Let $p_2(z)-p_1(z)$ be non-constant.\vspace{1.2mm}
\par Now we consider two possible subcases.\vspace{1.2mm}
\par \textbf{Subcase 2.1.} Let $p_2(z)+p_1(z)=\eta$, where $\eta$ is a constant in $\mathbb{C}$.\vspace{1.2mm}
\par Then, \eqref{e3.39} and \eqref{e3.40} yield  \bea\label{e2.46}\begin{cases}A_1e^{\eta}e^{p_1(z+c)-p_1(z)}+A_2e^{\eta}e^{-(p_1(z+c)+p_1(z))}-A_2e^{\eta}(1+H_1e^{\eta})e^{-2p_1(z)}\\=A_1(H_2+e^{\eta}),\\A_1e^{\eta}e^{p_1(z)-p_1(z+c)}+A_2e^{-\eta}e^{p_1(z+c)+p_1(z)}-A_2(H_3+e^{-\eta})e^{2p_1(z)}\\=A_1(H_4+e^{\eta}).\end{cases}\eea
\par Now, by similar argument as in Case 1, we can prove that $1+H_1e^{\eta}$, $H_2+e^{\eta}$, $H_3+e^{-\eta}$ and $H_4+e^{\eta}$ are all non-zero.\vspace{1.2mm}
\par Therefore, by Lemma \ref{lem3.1} and \eqref{e2.46}, we have \bea\label{e2.47} 	e^{p_1(z+c)-p_1(z)}=H_2e^{-\eta}+1\;\;\text{and}\;\;e^{p_1(z)-p_1(z+c)}=H_4e^{-\eta}+1.\eea
\par From \eqref{e2.46} and \eqref{e2.47}, we obtain \bea\label{e2.48}e^{-p_1(z+c)+p_1(z)}=1+H_1e^{\eta}\;\;\text{and}\;\;e^{p_1(z+c)-p_1(z)}=1+e^{\eta}H_3.\eea
\par Then by similar arguments as in the Case 1, we obtain that \beas p_1(z)=L(z)+\Phi_1(z_2)+\gamma \;\;\text{and}\;\; p_2(z)=-(L(z)+\Phi_1(z_2)+\gamma)+\eta.\eeas 
\par Therefore, from \eqref{e2.47} and \eqref{e2.48}, we obtain \bea\label{e2.49}\begin{cases}e^{L(c)}=1+e^{-\eta}\alpha^k,\;\;e^{-L(c)}=e^{-\eta}(-\alpha)^k+1,\\e^{-L(c)}=1+e^{\eta}(-\alpha)^k,\;\;e^{L(c)}=1+e^{\eta}\alpha^k.\end{cases}\eea
\par From first and fourth equations of \eqref{e2.49}, we get \beas e^{2\eta}=1.\eeas 
\par From the above two equations of \eqref{e2.49}, we obtain \bea\label{e2.50}(-1)^ke^{-\eta}\alpha^k+1+(-1)^k=0.\eea
\par If $k$ is odd, then in view of \eqref{e2.50}, we can get a contradiction. If $k$ is even, then from \eqref{e2.50}, we get \beas \alpha^k=-2e^{\eta}.\eeas
\par Therefore, from \eqref{e3.37}, we have \beas f(z)=\frac{A_1e^{L(z)+\Phi_1(z_2)+\gamma}+A_2e^{-(L(z)+\Phi_1(z_2)+\gamma)}}{\sqrt{2}\alpha^k},\eeas \beas g(z)=\frac{A_1e^{-(L(z)+\Phi_1(z_2)+\gamma)+\eta}+A_2e^{L(z)+\Phi_1(z_2)+\gamma-\eta}}{\sqrt{2}\alpha^k}.\eeas
\par \textbf{Subcase 2.2.} Let $p_2(z)+p_1(z)$ be a non-constant polynomial.\vspace{1.2mm}
\par Now we consider four possible cases below.
\par \textbf{Subcase 2.2.1.} Let $p_1(z+c)+p_2(z)=\zeta_1$ and $p_2(z+c)+p_1(z)=\zeta_2$, where $\zeta_1,\zeta_2\in\mathbb{C}$. Then, we easily see that $p_1(z+2c)-p_1(z)=\zeta_1-\zeta_2$ and $p_2(z+2c)-p_2(z)=\zeta_2-\zeta_1$. This implies that $p_1(z)=\alpha z+\beta_1$ and $p_2(z)=-\alpha z+\beta_2$, where $\alpha,\beta_1,\beta_2\in\mathbb{C}$. But, then $p_1(z)+p_2(z)=\beta_1+\beta_2=$ constant, a contradiction.\vspace{1.2mm}
\par \textbf{Subcase 2.2.2.} Let $p_1(z+c)+p_2(z)=\zeta_1\in\mathbb{C}$ and $p_2(z+c)+p_1(z)$ is non-constant.\vspace{1.2mm}
\par Then, by Lemma \ref{lem3.1} and \eqref{e3.40}, we get $A_2e^{-p_2(z+c)+p_1(z)}=A_1H_4$. This implies that $-p_2(z+c)+p_1(z)=\zeta_2\in\mathbb{C}$. Therefore, we must have $p_1(z+2c)+p_1(z)=\zeta_1+\zeta_2$, and hence $p_1(z)$ must be a constant, a contradiction.\vspace{1.2mm}
\par \textbf{Subcase 2.2.3.} Let $p_2(z+c)+p_1(z)=\zeta_1\in\mathbb{C}$ and $p_1(z+c)+p_2(z)$ is non-constant.\vspace{1.2mm}
\par Then, by similar argument as in subcase 2.2.2, we obtain a contradiction.\vspace{1.2mm}
\par \textbf{Subcase 2.2.4.} Let $p_2(z+c)+p_1(z)$ and $p_1(z+c)+p_2(z)$ both are non-constant.\vspace{1.2mm}
\par Then, using Lemma \ref{lem3.1}, we obtain from \eqref{e3.39} and \eqref{e3.40} that \beas A_2e^{-p_1(z+c)+p_2(z)}=A_1H_2\;\; \text{and}\;\; A_2e^{-p_2(z+c)+p_1(z)}=A_1H_4.\eeas
\par As $p_1(z)$ and $p_2(z)$ are non-constant polynomials, we must have $-p_1(z+c)+p_2(z)=\xi_1$ and $-p_2(z+c)+p_1(z)=\xi_2$, where $\xi_1,\xi_2\in\mathbb{C}$. This implies that $-p_1(z+c)+p_1(z)=-p_2(z+c)+p_2(z)=\xi_1+\xi_2$. Therefore, we must get $p_1(z)=\alpha z+\beta_1$ and $p_2(z)=\alpha z+\beta_2$, where $\alpha,\beta_1,\beta_2\in\mathbb{C}$. But, then we get $p_2(z)-p_1(z)=\beta_2-\beta_1=$constant, which is a contradiction.\end{proof}

\section{Statements and Declarations}
\vspace{1.3mm}

\noindent \textbf {Conflict of interest} The authors declare that there are no conflicts of interest regarding the publication of this paper.
\vspace{1.5mm}

\noindent{\bf Funding} There is no funding received from any organizations for this research work.
\vspace{1.5mm}

\noindent \textbf {Data availability statement}  Data sharing is not applicable to this article as no database were generated or analyzed during the current study.
\vspace{1.5mm}

\noindent{\bf Acknowledgment:} The authors would like to thank the referee(s) for the helpful suggestions and comments to improve the exposition of the paper.

\end{document}